\documentclass{daj}

\dajAUTHORdetails{%
  title = {A Geometric Simulation Theorem on Direct Products of Finitely Generated Groups}, %
  author = {Sebasti\'an Barbieri},
  plaintextauthor = {Sebastian Barbieri},
  keywords = {symbolic dynamics, effectively closed dynamical systems, simulation theorem, strongly aperiodic SFTs, branch groups.},
}   %

\dajEDITORdetails{%
   year={2019},
   number={9},
   received={30 July 2018},   %
   published={11 June 2019},  %
   doi={10.19086/da.8820},       %
}   %

\usepackage{graphicx,xcolor}
\usepackage{amsmath,amsthm}
\usepackage{amssymb}
\usepackage{float}
\usepackage[english]{babel}
\usepackage{cleveref}
\usepackage{pgfplots}

\theoremstyle{plain}
\newtheorem{theorem}{Theorem}[section]
\newtheorem{lemma}[theorem]{Lemma}
\newtheorem{corollary}[theorem]{Corollary}
\newtheorem{proposition}[theorem]{Proposition}
\newtheorem{definition}[theorem]{Definition}

\newtheorem{claim}[theorem]{Claim}

\theoremstyle{definition}
\newtheorem{remark}[theorem]{Remark}

\def\NN{\mathbb{N}}
\def\ZZ{\mathbb{Z}}

\def\ag{\mathcal{A}}

\def\CC{\mathcal{C}}

\def\FF{\mathcal{F}}
\def\Orb{\text{Orb}}
\def\Aut{\text{Aut}}

\def\supp{\operatorname{supp}}

\def\Topz{\texttt{Top}_{\texttt{1D}}}
\def\Grid{\texttt{Grid}}
\def\Topzz{\texttt{Top}_{\texttt{2D}}}

\def\Final{\texttt{Final}}

\newcommand{\define}[1]{\emph{#1}}
\newcommand{\isdef}{\triangleq}	

\usepackage{tikz}
\usetikzlibrary{patterns,positioning,arrows,decorations.markings,calc,decorations.pathmorphing,decorations.pathreplacing}
\usetikzlibrary{calc,intersections,through,backgrounds}
\usetikzlibrary{lindenmayersystems}
\usetikzlibrary{arrows}
\usepackage{tikz-3dplot}

\tikzstyle directedgreen=[postaction={decorate,decoration={markings,
		mark=at position .65 with {\arrow[scale=1]{latex}}}}]
\tikzstyle directedred=[postaction={decorate,decoration={markings,
		mark=at position .55 with {\arrow[scale=1]{latex}}}}]

\tdplotsetmaincoords{85}{97}

\definecolor{rouge}{RGB}{255,77,77}
\definecolor{vert}{RGB}{0,178,102}
\definecolor{jaune}{RGB}{255,255,0}
\definecolor{violet}{RGB}{208,32,144}
\definecolor{orange}{RGB}{255,140,0}
\definecolor{bleu}{RGB}{0,0,205}

\begin{document}

\begin{frontmatter}[classification=text]
\author[pgom]{Sebasti\'an Barbieri\thanks{Supported by the ANR project CoCoGro (ANR-16-CE40-0005)}}

\begin{abstract}
We show that every effectively closed action of a finitely generated group $G$ on a closed subset of $\{0,1\}^{\NN}$ can be obtained as a topological factor of the $G$-subaction of a $(G \times H_1 \times H_2)$-subshift of finite type (SFT) for any choice of infinite and finitely generated groups $H_1,H_2$. As a consequence, we obtain that every group of the form $G_1 \times G_2 \times G_3$ admits a non-empty strongly aperiodic SFT subject to the condition that each $G_i$ is finitely generated and has decidable word problem. As a corollary of this last result we prove the existence of non-empty strongly aperiodic SFT in a large class of branch groups, notably including the Grigorchuk group.
\end{abstract}
\end{frontmatter}

\section{Introduction}
An interesting approach when studying dynamical systems under the scope of computer science is to restrict to classes that can be described using a finite amount of information. Under this scope, a reasonably large and natural class is that of effectively closed systems. Informally speaking, effectively closed systems are those where both the configurations in the system and the group action can be described completely by a Turing machine. Even though these systems admit finite presentations, they remain quite complicated.

A natural question is whether an effectively closed system can be obtained as a subaction of another system which admits a simpler description. This question is motivated by the following: consider the class of subshifts of finite type (SFT), that is, the sets of colourings of a group which respect a finite number of local constraints --in the form of a finite list of forbidden patterns-- and are equipped with the shift action. It can be easily shown that any system obtained as a restriction of the shift action to a subgroup is not necessarily an SFT, but under weak assumptions, such as the groups being recursively presented, we obtain that the subaction is still an effectively closed dynamical system, see~\cite{Hochman2009b} for the $\ZZ^d$ case.

The previous situation has an analogue in the case of groups. It is known that any subgroup of a finitely presented group is recursively presented. A non-trivial theorem by Higman~\cite{Highman1961} shows that every recursively presented group can be embedded into a finitely presented group. This raises the question if it is possible to somehow ``embed'' an effectively closed dynamical system into a simpler class of systems, such as SFTs or sofic subshifts. We shall refer to such an embedding result as a \emph{simulation theorem}.

For the class of $\ZZ^d$-SFTs there is still no characterization of which dynamical systems can arise as their subactions, there are in fact some effective $\ZZ$-dynamical systems that cannot appear as a subaction of a $\ZZ^d$-SFT, or more generally, of any shift space, see Hochman~\cite{Hochman2009b}. Nevertheless, in that same article, it is proven that every effectively closed $\ZZ^d$-action on a closed subset $X \subset \{0,1\}^{\NN}$ admits an almost trivial isometric extension which can be realized as the subaction of a $\ZZ^{d+2}$-SFT. This simulation theorem has subsequently been improved for the expansive case independently in \cite{AubrunSablik2010} and \cite{DurandRomashchenkoShen2010} showing that every effectively closed $\ZZ$-subshift can in fact be obtained as the projective subdynamics of a sofic $\ZZ^2$-subshift. 

Simulation theorems are powerful tools to prove properties about the original systems. A notorious example is the characterization of the set of entropies of $\ZZ^2$-SFTs as the set of right recursively enumerable numbers~\cite{HochmanMeyerovitch2010}. More recently, in an article of Sablik and the author~\cite{BarbieriSablik2017}, Hochman's simulation theorem was extended to groups which are of the form $G = \ZZ^d \rtimes H$ for $d > 1$. More specifically, it was shown that for every finitely generated group $H$, homomorphism $\varphi\colon H \to \Aut(\ZZ^d)$ and effectively closed $H$-dynamical system $(X,T)$ one can construct a $(\ZZ^d \rtimes_{\varphi}  H)$-SFT whose $H$-subaction is an extension of $(X,T)$. As a consequence of this result, we showed that groups of the form $\ZZ^d \rtimes_{\varphi} H$ admit strongly aperiodic SFTs whenever the word problem of $H$ is decidable.

All of the previous simulation theorems have as common denominator the employment of a $\ZZ$ or $\ZZ^2$ component to simulate space-time diagrams of Turing machines. A natural question would be to ask whether it is possible to obtain a simulation theorem which involves exclusively periodic groups, that is, groups for which $\ZZ$ does not embed.

The purpose of this article is to prove a simulation theorem which does not necessarily involve a $\ZZ$ component. More precisely,

{
	\renewcommand{\thetheorem}{\ref{simulationTHEOREM}}
	\begin{theorem}
		Let $G$ be a finitely generated group and $(X,T)$ an effectively closed $G$-dynamical system. For every pair of infinite and finitely generated groups $H_1,H_2$ there exists a $(G \times H_1 \times H_2)$-SFT whose $G$-subaction is an extension of $(X,T)$.
	\end{theorem}
	\addtocounter{theorem}{-1}
}

This result is obtained through the combination of two different techniques already present in the literature. On the one hand, we use Toeplitz configurations to encode the dynamical system $(X,T)$ in an effectively closed $\ZZ$-subshift and then extend this object to a $\ZZ^2$-sofic subshift through the theorems of~\cite{AubrunSablik2010,DurandRomashchenkoShen2010}. This is the main idea employed in~\cite{BarbieriSablik2017}. On the other hand, we make use of a technique by Jeandel~\cite{Jeandel2015_translation} to force grid structures through local rules. Namely, a theorem by Seward~\cite{Seward2014} shows that a geometric analogue of Burnside's conjecture holds, namely, that every infinite and finitely generated group admits a translation-like action by $\ZZ$. From a graph theoretical perspective, this means that the group admits a set of generators such that its associated Cayley graph can be covered by disjoint bi-infinite paths. We use that theorem and Jeandel's technique to geometrically embed a two-dimensional grid into $H_1 \times H_2$ and create the necessary structure to prove our main result.

In the case where the $G$-dynamical system is expansive, we can give a stronger result.

{
	\renewcommand{\thetheorem}{\ref{theorem_expansive_case}}
	\begin{theorem}
		Let $G$ be a recursively presented and finitely generated group and $Y$ an effectively closed $G$-subshift. For every pair of infinite and finitely generated groups $H_1,H_2$ there exists a sofic $(G \times H_1 \times H_2)$-subshift $X$ such that
		\begin{itemize}
			\item the $G$-subaction of $X$ is conjugate to $Y$.
			\item the $G$-projective subdynamics of $X$ is $Y$.
			\item The shift action $\sigma$ restricted to $H_1 \times H_2$ is trivial on $X$.
		\end{itemize}
	\end{theorem}
	\addtocounter{theorem}{-1}
}

It is known that every non-empty $\ZZ$-SFT contains a periodic configuration~\cite{lind1995introduction}. More generally, any SFT defined over a finitely generated free group also has a periodic configuration~\cite{Piantadosi2008}. However, it was shown by Berger~\cite{Berger1966} that there are $\ZZ^2$-SFTs which are strongly aperiodic, that is, such that the shift acts freely on the set of configurations. This result has been proven several times with different techniques~\cite{Robinson1971,Kari1996259,JeandelRao11Wang} giving a variety of constructions. However, the problem of determining which is the class of finitely generated groups which admit strongly aperiodic SFTs remains open. Amongst the groups that do admit strongly aperiodic SFTs are: $\ZZ^d$ for $d > 1$, hyperbolic surface groups~\cite{CohenGoodmanS2015} and more generally one-ended word-hyperbolic groups~\cite{CohenGoodmanRieck2017}, Osin and Ivanov monster groups~\cite{Jeandel2015}, the direct product $G \times \ZZ$ for a particular class of groups $G$ which includes Thompson's group $T$ and $\text{PSL}(2,\ZZ)$~\cite{Jeandel2015} and groups of the form $\ZZ^d \rtimes G$ where $d >1$ and $G$ is a finitely generated group with decidable word problem~\cite{BarbieriSablik2017}. It is also known that no group with two or more ends can support strongly aperiodic SFTs~\cite{Cohen2014} and that recursively presented groups which admit strongly aperiodic SFTs must have decidable word problem~\cite{Jeandel2015}. 

As an application of \Cref{theorem_expansive_case} we present a new class of groups which admit strongly aperiodic SFTs. 

{
	\renewcommand{\thetheorem}{\ref{corollaryaperioci}}
	\begin{theorem}
		For any triple of infinite and finitely generated groups $G_1,G_2$ and $G_3$ with decidable word problem, their direct product $G_1 \times G_2 \times G_3$ admits a non-empty strongly aperiodic subshift of finite type.
	\end{theorem}
	\addtocounter{theorem}{-1}
}

A result by Carroll and Penland~\cite{CarrollPenland} shows that having a non-empty strongly aperiodic subshift of finite type is a commensurability invariant of groups. Putting this together with \Cref{corollaryaperioci} we deduce that any finitely generated group with decidable word problem which is commensurable to its square also has the property. In particular, as the Grigorchuk group is commensurable to its square, this yields the existence of non-empty strongly aperiodic subshifts of finite type in the Grigorchuk group.

{
	\renewcommand{\thetheorem}{\ref{cor_grigo}}
	\begin{corollary}
		There exists a non-empty strongly aperiodic subshift of finite type defined over the Grigorchuk group.
	\end{corollary}
	\addtocounter{theorem}{-1}
}

This strengthens Jeandel's result from~\cite{Jeandel2015_translation} where the Grigorchuk group was shown to admit a weakly aperiodic SFT, that is, a subshift such that the orbit of every configuration under the shift action is infinite. More generally, we show that the same result holds for any finitely generated branch group with decidable word problem and in fact characterizes recursively presented branch groups with decidable word problem.

{
	\renewcommand{\thetheorem}{\ref{teorema_branch}}
	\begin{theorem}
		Let $G$ be a finitely generated and recursively presented branch group. Then $G$ has decidable word problem if and only if there exists a non-empty strongly aperiodic $G$-SFT.
	\end{theorem}
	\addtocounter{theorem}{-1}
}

\section{Preliminaries}

Consider a group $G$ and a compact topological space $X$. The tuple $(X,T)$ where $T\colon G \curvearrowright X$ is a left $G$-action by homeomorphisms $(T^g)_{g \in G}$ is called a \define{$G$-dynamical system}. Let $(X,T)$ and $(Y,S)$ be two $G$-dynamical systems. We say $\phi\colon X \to Y$ is a \define{topological morphism} if it is continuous and $G$-equivariant, that is, $\phi \circ T^g = S^g \circ \phi$ for all $g \in G$. A surjective topological morphism $\phi \colon X \twoheadrightarrow Y$ is a \define{topological factor} and we say that  $(Y,S)$ is a \define{factor} of $(X,T)$ and that $(X,T)$ is an \define{extension} of $(Y,S)$. When $\phi$ is a bijection and its inverse is continuous we say it is a  \define{topological conjugacy} and that $(X,T)$ is \define{conjugated} to $(Y,S)$.

In what follows, we consider the space $X$ to be a closed subset of $\{0,1\}^{\NN}$ equipped with the product topology and $G$ to be a finitely generated group with identity $1_G$. For a word $w = w_0w_1\dots w_n \in \{0,1\}^* \isdef \bigcup_{k \in \NN}\{0,1\}^k$ we denote by $[w]$ the set of all $x \in \{0,1\}^{\NN}$ such that $x_i = w_i$ for $i \leq n$. That is, $[w]$ is the set of all infinite binary words which start with $w$. 

\begin{definition}
	A closed subset $X \subset \{0,1\}^{\NN}$ is \define{effectively closed} if there exists a Turing machine which on entry $w \in \{0,1\}^*$ accepts if and only if $[w] \cap X = \emptyset$.
\end{definition}

In other words, an effectively closed set $X$ is one for which there exists a Turing machine which enumerates its complement as an effective open set. This gives a sequence of upper approximations of $X$ whose intersection is $X$.

\begin{definition}
	Let $S$ be a fixed finite set of generators of $G$. A group action $T\colon G \curvearrowright X$ is \define{effectively closed} if there exists a Turing machine which on entry $s \in S$ and $v,w \in \{0,1\}^*$ accepts $(s,v,w)$ if and only if $[v] \cap T^s([w] \cap X)  = \emptyset$.
\end{definition}

An effectively closed action can also be understood in the following manner. There is a Turing machine which on entry $s,w$ enumerates a sequence of words $(v_j)_{j \in J}$ such that $T^{s}([w] \cap X) = \{0,1\}^{\NN} \setminus \bigcup_{j \in J}[v_j]$. In other words, it yields a sequence of upper approximations of $T^{s}([w]\cap X)$ whose intersection is $T^{s}([w]\cap X)$.

\begin{remark} It is not hard to show that whenever $T\colon G \curvearrowright X$ is effectively closed, one can construct a Turing machine which on entry $s_1\dots s_n \in S^*$ and $v,w \in \{0,1\}^*$ accepts if and only if $[v] \cap T^{s_1\dots s_n}([w] \cap X) = \emptyset$. In particular, this implies that the definition of effectively closed action does not depend upon the choice of generators $S$.
\end{remark}

\begin{remark} If $T\colon G \curvearrowright X$ is effectively closed then $X$ is also effectively closed. Indeed, it suffices to choose $w = \epsilon$ to be the empty word and fix $s \in S$. Let $v \in \{0,1\}^*$, by definition, the machine accepts $(s,v,\epsilon)$ if and only if $[v] \cap X = \emptyset$. 
\end{remark}

\begin{remark}
	The definition of effectively closed action does not make any assumption on the recursive properties of the acting group. For instance, the trivial action over $\{0,1\}^{\NN}$ is effectively closed for any finitely generated group.
\end{remark}

\begin{definition}
	Let $S$ be a fixed finite set of generators of $G$. We say a $G$-dynamical system $(X,T)$ is \define{effectively closed} if $T \colon G \curvearrowright X$ is an effectively closed action.
\end{definition}

The pointwise idea behind this definition is that there is an algorithm which given a word $s_1\dots s_k \in S^*$ representing an element $g \in G$ and the first $n$ coordinates of $x \in X \subset \{0,1\}^{\NN}$ returns an approximation of $T^g(x)$. In other words, as we increase $n$ and let the machine run for longer periods of time, we get better approximations of $T^g(x)$.

Let $\ag$ be a finite alphabet and $G$ a finitely generated group. The set $\ag^G = \{ x\colon G \to \ag\}$ equipped with the left group action $\sigma\colon G \times \ag^G \to \ag^G$ given by: 
$(\sigma^{h}(x))(g) \isdef x(h^{-1}g)$ is the \textit{full $G$-shift}. The elements $a \in \ag$ and $x \in \ag^G$ are called \define{symbols} and \define{configurations} respectively. We endow $\ag^G$ with the product topology, therefore obtaining a compact metric space. The topology is generated by the metric $\displaystyle{d(x,y) = 2^{-\inf\{|g|\; \mid\; g \in G:\; x(g) \neq y(g)\}}}$ where $|g|$ is the length of the smallest expression of $g$ as the product of some fixed set of generators of $G$. This topology is also generated by the clopen subbase given by the \define{cylinders} $[a]_g = \{x \in \ag^G~|~x(g) = a\in \ag\}$. A \emph{support} is a finite subset $F \subset G$. Given a support $F$, a \emph{pattern with support $F$} is an element $p \in \ag^F$, i.e. a finite configuration and we write $\operatorname{supp}(p) = F$. Analogously to words, we denote the cylinder generated by $p$ centered in $g$ as $[p]_g = \bigcap_{h \in F}[p(h)]_{gh}$. If $x \in [p]_g$ for some $g \in G$ we write $p \sqsubset x$ to say that $p$ \emph{appears} in $x$.

A subset $X \subset \ag^G$ is a \define{$G$-subshift} if and only if it is $\sigma$-invariant (for each $g \in G$, $\sigma^g(X) \subset X$) and closed in the product topology. Equivalently, $X$ is a $G$-subshift if and only if there exists a set of forbidden patterns $\FF$ such that
$$X=X_\FF \isdef  {\ag^G \setminus \bigcup_{p \in \FF, g \in G} [p]_g}.$$

Said otherwise, the $G$-subshift $X_{\FF}$ is the set of all configurations $x \in \ag^G$ such that no $p \in \FF$ appears in $x$.

If the context is clear enough, we will drop the group $G$ from the notation and simply refer to a subshift. We shall also use the notation $\ag_X$ to denote the alphabet of $X$ and $\FF_X$ to denote some set of forbidden patterns such that $X = X_{\FF_X}$. A subshift $X\subseteq \ag^G$ is \define{of finite type} -- SFT for short -- if there exists a finite set of forbidden patterns $\FF$ such that $X=X_\FF$. A subshift $X\subseteq \ag^G$ is \define{sofic} if it is the factor of an SFT. Finally, a subshift is \define{effectively closed} if there exists a recursively enumerable coding of a set of forbidden patterns $\FF$ such that $X=X_\FF$. More details can be found in~\cite{ABS2017}. In the case of $\ZZ$-subshifts, we say $X$ is effectively closed if and only if there exists a recursively enumerable set of forbidden words $\FF$ such that $X = X_{\FF}$.

We say a $\ZZ^2$-subshift is \emph{nearest neighbour} if there exists a set of forbidden patterns $\FF$ defining it such that each $p \in \FF$ has support $\{(0,0), (1,0)\}$ or $\{(0,0), (0,1)\}$. While there are $\ZZ^2$-SFTs which are not nearest neighbour, each $\ZZ^2$-SFT is topologically conjugate to a nearest neighbour $\ZZ^2$-subshift through a higher block recoding, see for instance~\cite{lind1995introduction} for the 1-dimensional case. 

\begin{remark}\label{remark_nn_and_1_block}
	For every sofic $\ZZ^2$-subshift $Y$ we can jointly extract a nearest neighbour $\ZZ^2$-SFT extension $\widehat{X}$ and a $1$-block topological factor $\widehat{\phi}\colon \widehat{X} \twoheadrightarrow Y$, that is, a topological factor such that there exists a local recoding of the alphabet $\widehat{\Phi} \colon \ag_{\widehat{X}} \to \ag_{Y}$ such that for each $x \in \widehat{X}$ and $z \in \ZZ^2$ we have $(\widehat{\phi}(x))(z) = \widehat{\Phi}(x(z))$. This fact follows from the Curtis--Lyndon--Hedlund theorem and a higher-block recoding. For a proof in general finitely generated groups, see Proposition 1.3 and Proposition 1.6 of~\cite{BarbieriThesis}.
\end{remark}

Let $H \leq G$ be a subgroup and $(X,T)$ a $G$-dynamical system. The \emph{$H$-subaction} of $(X,T)$ is $(X,T_H)$ where $T_H \colon H \curvearrowright X$ is the restriction of $T$ to $H$, that is, for each $h \in H$, then $(T_H)^h(x) = T^h(x)$. In the case of a subshift $X \subset \ag^G$ there is also the different notion of projective subdynamics. The \emph{$H$-projective subdynamics} of $X$ is the set $\pi_H(X) = \{ y \in \ag^H \mid \exists x \in X, \forall h \in H, y(h) = x(h) \}$. It is important to remark that subactions do not preserve expansivity, so in particular a subaction of a subshift is not necessarily a subshift. On the contrary, the projective subdynamics of a subshift $\pi_H(X)$ is always an $H$-subshift.

\section{Simulation without an embedded copy of $\ZZ$}
\label{section.generalities}

The purpose of this section is to prove the following result.

\begin{theorem}\label{simulationTHEOREM}
	Let $G$ be a finitely generated group and $(X,T)$ an effectively closed $G$-dynamical system. For every pair of infinite and finitely generated groups $H_1,H_2$ there exists a $(G \times H_1 \times H_2)$-SFT whose $G$-subaction is an extension of $(X,T)$.
\end{theorem}

The general scheme of the proof is the following: first, in \Cref{subsectop} we construct a $\ZZ$-subshift $\Topz(X,T)$ which encodes an arbitrary action $T \colon G \curvearrowright X$ of a finitely generated group through the use of Toeplitz sequences. We show that whenever $T \colon G \curvearrowright X$ is effectively closed, the $\ZZ$-subshift $\Topz(X,T)$ is effectively closed as well. Subsequently, we extend $\Topz(X,T)$ to a $\ZZ^2$-subshift by repeating its rows periodically in the vertical direction. Using a known simulation theorem~\cite{AubrunSablik2010,DurandRomashchenkoShen2010} we conclude that this two-dimensional subshift, which we denote by $\Topzz(X,T)$, is a sofic $\ZZ^2$-subshift from which we extract a nearest neighbour SFT extension. 

The next step is presented in \Cref{subsecgrid} where we construct an $(H_1 \times H_2)$-SFT $\Grid$ with the property that each $\omega \in \Grid$ induces a bounded $\ZZ^2$-action on $H_1 \times H_2$. We use a result by Seward~\cite{Seward2014} to guarantee that for a specific choice of generators of $H_1$ and $H_2$ there is at least one $\omega \in \Grid$ inducing a free action. We will use these $\omega$ as replacements of two-dimensional grids and to embed in them configurations of a nearest neighbour $\ZZ^2$-subshift.

Finally, in \Cref{subsecproof} we use the simulated grids in $H_1 \times H_2$ to encode a nearest neighbour extension of $\Topzz(X,T)$. This yields an $(H_1 \times H_2)$-SFT which factors onto a sofic $(H_1 \times H_2)$-subshift where every grid codes an element $x \in X$ and its image under $T$ by the generators of $G$. We then proceed to extend this object to a $(G \times H_1 \times H_2)$-SFT $\Final(X,T)$ by forcing every $(H_1 \times H_2)$-coset to have exactly the same grid structure and by linking them through local rules which mimic the action $T \colon G \curvearrowright X$ along every generator of $G$. We finish this last step by defining the topological factor and showing that it satisfies the required properties.

\subsection{Encoding an effectively closed dynamical system using Toeplitz configurations}\label{subsectop}

\newcommand{\cuadrito}{
	\begin{tikzpicture}
	\draw [fill = black!50] (0,0) rectangle +(0.25,0.25);
	\end{tikzpicture}}

Let $T \colon G \curvearrowright X$ be an action of a finitely generated group. Here we show how to encode $T$ into a $\ZZ$-subshift. The ideas presented in here are very similar as those of Section 3.2 of~\cite{BarbieriSablik2017}, although here we shall only treat a special simplified case which suffices for our purposes.

A configuration $\tau \in \ag^{\ZZ}$ is said to be \emph{Toeplitz} if for every $m \in \ZZ$ there is $p > 0$ such that $\tau(m) = \tau(m+kp)$ for each $k \in \ZZ$. These configurations were initially defined by Jacobs and Keane~\cite{JacobsKeane1969} for one-sided dynamical systems and are quite useful to encode information in a recurrent way. Indeed, consider the function $\Psi \colon \{0,1\}^{\NN} \to \{0,1,\cuadrito\}^{\ZZ}$ given by:

$$ \Psi(x)(j) \isdef \begin{cases}
x_n \mbox{\ \ \  if } j = 3^n \mod{3^{n+1}} \\ \cuadrito  \mbox{ \ \ \ \ in the contrary case. }
\end{cases} $$

For instance, if $x = x_0x_1x_2x_3\dots$ we obtain that $\Psi(x)$ is

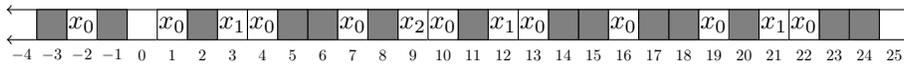
\begin{center}
	\begin{tikzpicture}[scale = 0.4]
	\draw [<->] (-4,1) -- (26,1);
	\draw [<->] (-4,0) -- (26,0);
	\foreach \i in {-4,...,25}{
		\node at (\i+0.5,-0.5) {\scalebox{0.6}{$\i$}};
	}
	\draw [fill = black!50] (18,0) rectangle +(1,1);
	\draw [fill = white] (9,0) rectangle +(1,1);
	\node at (9.5,0.5) {$x_{2}$};
	\draw [fill = black!50] (-3,0) rectangle +(1,1);
	\draw [fill = black!50] (0,0) rectangle +(1,1);
	
	\foreach \i in {0,...,2}{
		\draw [fill = black!50] (9*\i+6,0) rectangle +(1,1);
		\draw [fill = white] (9*\i+3,0) rectangle +(1,1);
		\node at (9*\i+3.5,0.5) {$x_{1}$};
	}
	\foreach \i in {-1,...,7}{
		\draw [fill = black!50] (3*\i+2,0) rectangle +(1,1);
		\draw [fill = white] (3*\i+1,0) rectangle +(1,1);
		\node at (3*\i+1.5,0.5) {$x_{0}$};
	}
	\end{tikzpicture}
\end{center}

Technically speaking, $\Psi(x)$ is not Toeplitz as $m = 0$ fails to satisfy the requirement, however, every other $m \in \ZZ \setminus \{0\}$ does. For $x = (x_i)_{i \in \NN} \in \{0,1\}^{\NN}$ let $\sigma(x) \in \{0,1\}^{\NN}$ be the one-sided shift defined by $\sigma(x)_i=x_{i+1}$ and note that for every $j \in \ZZ$ we have that:
$$\Psi(x)(3j) = \Psi(\sigma(x))(j), \ \ \ \Psi(x)(3j+1) = x_0, \ \mbox{and} \ \Psi(x)(3j+2) = \cuadrito.$$ 
Let $\Orb(\Psi(x))$ be the two-sided orbit of $\Psi(x)$. The important property of $\Psi(x)$ is that $x$ can be recognized locally from any configuration $y \in \overline{\Orb(\Psi(x))}$. Indeed, each subword of length $3$ in $\Psi(x)$ is a cyclic permutation of a word of the form $ax_0\cuadrito$ where $a \in \{0,1,\cuadrito\}$, therefore $x_0$ can be recognized as it is the only non-$\cuadrito$ symbol which is followed by a $\cuadrito$. Similarly, any word of length $9$ can be used to decode $x_0,x_1$ and generally, a word of length $3^n$ is sufficient to decode $x_0,x_1,\dots x_{n-1}$. As any $y \in \overline{\Orb(\Psi(x))}$ must coincide in arbitrarily large blocks with a shift of $\Psi(x)$ we have that this property holds for every configuration in $\overline{\Orb(\Psi(x))}$.

Let $(X,T)$ be a $G$-dynamical system. We use the encoding $\Psi$ defined above to construct an effectively closed $\ZZ$-subshift $\Topz(X,T)$ which encodes the configurations of $X$ and their images under the action of $T$ along the generators $S$. Formally, let $S \subset G$ be a finite and symmetric $(S^{-1} \subset S)$ set of generators of $G$ which contains the identity $1_G$. 

We define $\Topz(X,T)$ as the $\ZZ$-subshift over the alphabet $\{0,1,\cuadrito\}^{S}$ given by the set of forbidden words $\FF_{\Topz(X,T)}$, where $\FF_{\Topz(X,T)} = \bigcup_{n \in \NN}\FF_n$ and $\FF_n$ is the set of words $w$ of length $3^{n+1}$ over the alphabet $\{0,1,\cuadrito\}^S$ which are accepted by the following procedure.

\textbf{Procedure:} given $w \in (\{0,1,\cuadrito\}^S)^{3^{n+1}}$ and $s \in S$ denote by $w(s) \in \{0,1,\cuadrito\}^{3^{n+1}}$ the $s$th component of $w$. For each $s \in S$ do the following: fix $j \isdef 0$, $v \isdef w(s)$ and check whether there exists $b \in \{0,1\}$ and $(a_i)_{1 \leq i \leq 3^{n-j}}$ with $a_i \in \{0,1,\cuadrito\}$ such that $v$ is a cyclic permutation of the concatenation of all $a_ib\ \cuadrito$. That is \begin{align*}
v \in \{ & a_1b\ \cuadrito\ a_2 b\ \cuadrito\ \dots a_{3^{n-j}}b\  \cuadrito\ , \\
& \cuadrito\ a_1b\ \cuadrito\ a_2 b\ \cuadrito\ \dots a_{3^{n-j}}b\ ,\\
& b\ \cuadrito\ a_1b\ \cuadrito\ a_2 b\ \cuadrito\ \dots a_{3^{n-j}}\}.
\end{align*}

If no such $b$ exists, declare $w \in \FF_n$, otherwise, define $u(s)_j \isdef b$, increase the counter $j \isdef j+1$ and repeat the procedure with $v \isdef a_1a_2 \dots a_{3^{n-j}}$. Repeat this procedure until $j \isdef n+1$. If at some iteration of this procedure no such $b$ exists, declare $w \in \FF_n$. Otherwise, we obtained a symbol $u(s)_j$ for each $j \in \{0,\dots,n\}$. We illustrate this procedure in~\Cref{figure_algoritmo}.
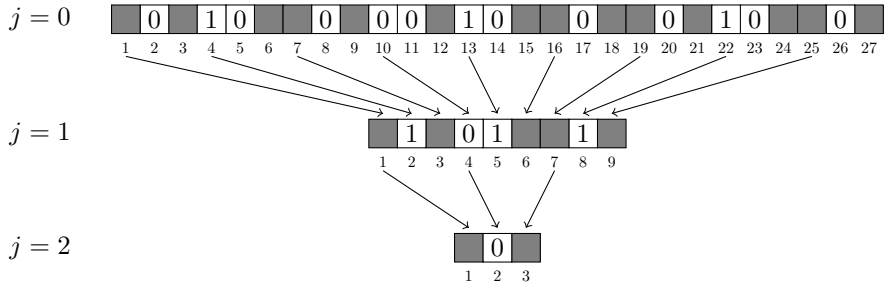
\begin{figure}[h!]
	\begin{tikzpicture}[scale = 0.38]
	\foreach \i in {1,...,27}{
		\node at (\i-0.5,-0.5) {\scalebox{0.6}{$\i$}};
	}
	\foreach \i in {0,...,8}{
		\draw [fill = black!50] (3*\i+2,0) rectangle +(1,1);
		\draw [fill = white] (3*\i+1,0) rectangle +(1,1);
		\node at (3*\i+1.5,0.5) {$0$};
		\draw [->] (3*\i+0.5,-0.8) -- (9.5+\i,-2.8);
	}
	\foreach \i in {0,...,2}{
		\draw [fill = black!50] (9*\i+6,0) rectangle +(1,1);
		\draw [fill = white] (9*\i+3,0) rectangle +(1,1);
		\node at (9*\i+3.5,0.5) {$1$};
	}
	\draw [fill = black!50] (0,0) rectangle +(1,1);
	\draw [fill = black!50] (18,0) rectangle +(1,1);
	\draw [fill = white] (9,0) rectangle +(1,1);
	\node at (9.5,0.5) {$0$};
	\node at (-2.5,0.5) {$j = 0$};
	\node at (-2.5,-3.5) {$j = 1$};
	\node at (-2.5,-7.5) {$j = 2$};
	\begin{scope}[shift = {(9,-4)}]
	\foreach \i in {1,...,9}{
		\node at (\i-0.5,-0.5) {\scalebox{0.6}{$\i$}};
	}
	\foreach \i in {0,...,2}{
		\draw [fill = black!50] (3*\i+2,0) rectangle +(1,1);
		\draw [fill = white] (3*\i+1,0) rectangle +(1,1);
		\node at (3*\i+1.5,0.5) {$1$};
		\draw [->] (3*\i+0.5,-0.8) -- (3.5+\i,-2.8);
	}
	\draw [fill = black!50] (0,0) rectangle +(1,1);
	\draw [fill = black!50] (6,0) rectangle +(1,1);
	\draw [fill = white] (3,0) rectangle +(1,1);
	\node at (3.5,0.5) {$0$};
	
	\end{scope}
	\begin{scope}[shift = {(12,-8)}]
	\foreach \i in {1,...,3}{
		\node at (\i-0.5,-0.5) {\scalebox{0.6}{$\i$}};
	}
	\draw [fill = black!50] (0,0) rectangle +(1,1);
	\draw [fill = black!50] (2,0) rectangle +(1,1);
	\draw [fill = white] (1,0) rectangle +(1,1);
	\node at (1.5,0.5) {$0$};
	\end{scope}
	\end{tikzpicture}
	\caption{The procedure applied to a word $w(s)$ of length $27 = 3^{2+1}$. In this case, the word is not rejected and $u(s) = 010$ is decoded.}
	\label{figure_algoritmo}
\end{figure}

If this stage of the procedure is reached, we have for each $s \in S$ a word $u(s) \isdef u(s)_0u(s)_1\dots u(s)_n \in \{0,1\}^{n+1}$. We declare $w$ to be in $\FF_n$ if for some $s \in S$ we have that either $[u(s)] \cap  X = \emptyset$ or $[u(s)] \cap T^s([u(1_G)] \cap X) = \emptyset$.

\begin{remark}
	There is at most one way to cyclically partition a word $w(s) \in \{0,1,\cuadrito\}^{3^{n+1}}$ in segments of the form $a_i b\ \cuadrito$ as described before. Thus if the procedure produces words $(u(s))_{s \in S}$ these are unique.
\end{remark}

\begin{proposition}\label{propo_effective_subshift}
	If $(X,T)$ is an effectively closed $G$-dynamical system, then $\Topz(X,T)$ is an effectively closed $\ZZ$-subshift.
\end{proposition}

\begin{proof}
	The first part of the procedure can easily be implemented by a Turing machine. The existence of an algorithm which accepts the list of words $(u(s))_{s \in S}$ if and only if for some $s \in S$ either $[u(s)] \cap  X = \emptyset$ or $[u(s)] \cap T^s([u(1_G)] \cap X) = \emptyset$ is given by the definition of effectively closed action. It suffices to run that algorithm in parallel for every $s \in S$. This shows that $\FF_{\Topz(X,T)} = \bigcup_{n \in \NN} \FF_n$ is a recursively enumerable language and hence $\Topz(X,T)$ is effectively closed.
\end{proof}

Let $y \in \Topz(X,T)$ and denote its $s$th component by $y(s)$. By definition, we have that for each $n > 0$ none of the words of length $3^{n}$ appearing in $y$ belong to $\FF_{\Topz(X,T)}$ and thus any $w(s) \sqsubset y(s)$ of length $3^{n}$ defines a unique word $u_w(s) \in \{0,1\}^n$ by the procedure. Moreover, one can easily verify with the definition of $\FF_{\Topz(X,T)}$ by using the words of length $3^{n+1}$ that $u_w(s)$ does not depend on the specific choice of $w(s) \sqsubset y(s)$ but only on $y(s)$, and thus we  can define a family of functions $\gamma_n \colon S \times \Topz(X,T) \to \{0,1\}$ by $$\gamma_n(s,y) \isdef (u_{y|_{\{0,\dots 3^{n+1}-1\}}}(s))_n.$$

Said otherwise, $\gamma_n$ recovers the $n$th symbol coded by the $s$th component of $y \in \Topz(X,T)$. Furthermore, we can use the $\gamma_n$ to construct a function $\gamma \colon S \times \Topz(X,T) \to \{0,1\}^{\NN}$ defined by $\gamma(s,y)_n \isdef \gamma_n(s,y).$ By definition of $\FF_{\Topz(X,T)}$ we have that for each $n \in \NN$ then:
$$[\gamma(s,y)_0,\dots,\gamma(s,y)_n] \cap X \neq \emptyset,$$
$$[\gamma(s,y)_0,\dots,\gamma(s,y)_n] \cap T^s([\gamma(1_G,y)_0,\dots,\gamma(1_G,y)_n] \cap X) \neq \emptyset.$$

From the fact that $X$ is compact and $T$ continuous, we deduce that $\gamma(s,y) \in X$ and $T^s(\gamma(1_G,y)) = \gamma(s,y)$. Also, if for $x \in X$ we define $\widetilde{\Psi}(x)$ as the configuration in $(\{0,1,\cuadrito\}^{S})^{\ZZ}$ such that $(\widetilde{\Psi}(x)(j))(s) \isdef \Psi(T^s(x))(n)$ we can verify that $\widetilde{\Psi}(x) \in \Topz(X,T)$ and $\gamma(s,\widetilde{\Psi}(x)) = T^s(x)$. Therefore $\gamma\colon S \times \Topz(X,T) \to X$ is onto. Finally, from the fact that each $\gamma_n$ only depends on an arbitrary subword of length $3^n$ of $y$, we obtain that $\gamma(s,y) = \gamma(s,\sigma^m(y))$ for each $m \in \ZZ$ and that $\gamma$ is continuous.

Next we will make use of a known simulation theorem to lift our effectively closed $\ZZ$-subshift up to a sofic $\ZZ^2$-subshift.

\begin{theorem}[\cite{AubrunSablik2010,DurandRomashchenkoShen2010}]\label{theorem_simulacion_natalie}
	If $X$ is an effectively closed $\ZZ$-subshift, then the $\ZZ^2$-subshift $Y$ for which every $y \in Y$ satisfies $\sigma^{(0,1)}(y) = y$ and $\pi_{(\ZZ,0)}(Y) = X$ is sofic.
\end{theorem}

Using \Cref{theorem_simulacion_natalie} we obtain a sofic $\ZZ^2$-subshift which we call $\Topzz(X,T)$. As every row in a configuration in $\Topzz(X,T)$ is the same, we can naturally extend the definition of $\gamma$ to this subshift by restricting to $\ZZ \times \{0\}$. We summarize the important points of all that has been constructed in this subsection in the following lemma.

\begin{lemma}\label{lemma_allforone}
	There exists a sofic $\ZZ^2$-subshift $\Topzz(X,T)$ and a continuous function $\gamma \colon S \times \Topzz(X,T) \to X$ with the following properties:
	\begin{enumerate}
		\item For each $s \in S$, $y \mapsto \gamma(s,y)$ is onto.
		\item For any  $z \in \ZZ^2$, $\gamma(s,y) = \gamma(s, \sigma^z(y)).$
		\item $\gamma(s,y)= T^s(\gamma(1_G,y))$.
	\end{enumerate}
\end{lemma}

\subsection{Finding a grid in $H_1 \times H_2$}\label{subsecgrid}

Let $H_1,H_2$ be infinite and finitely generated groups. The aim of this section is to construct an $(H_1 \times H_2)$-SFT which emulates copies of $\ZZ^2$. If both $H_1$ and $H_2$ had a non-torsion element this would be an easy task, however, this is not true in general. We bypass this restriction by using the notion of translation-like action introduced by Whyte~\cite{Whyte1999}. Instead of having the rigidity of a proper translation, translation-like actions only ask for the action to be free and such that each element of the acting group moves elements of the set a uniformly bounded distance away. Formally,

\begin{definition}
	A left action of a group $G$ over a metric space $(X,d)$ is translation-like if and only if it satisfies:
	\begin{enumerate}
		\item $G \curvearrowright X$ is free, that is, $gx = x$ implies $g = 1_G$.
		\item For every $g \in G$ the set $\{d(gx,x) \mid x \in X\}$ is bounded.
	\end{enumerate}
\end{definition}

This notion gives a proper setting to define geometric analogues of classical disproved conjectures in group theory concerning subgroup containments. For instance, the Burnside conjecture and the Von Neumann conjecture can be reinterpreted geometrically as the question of whether every infinite and finitely generated group admits a translation-like action by $\ZZ$ or by a non-abelian free group respectively.

In what concerns our study, we are only going to make use of the following result from Seward~\cite{Seward2014} which is the positive answer to the geometric version of the Burnside conjecture:

\begin{theorem}[\cite{Seward2014},~Theorem~1.4]\label{teorema_seward}
	Every finitely generated infinite group admits a translation-like action of $\ZZ$.
\end{theorem}

Let us remark that~\Cref{teorema_seward} has already been used by Jeandel in~\cite{Jeandel2015_translation} to show that the Domino problem --that is, the problem of deciding whether a finite set of forbidden patterns defines a non-empty subshift-- is undecidable for groups of the form $H_1 \times H_2$ where both $H_i$ are infinite and finitely generated. He also showed that groups containing such a product as a subgroup have undecidable Domino problem and admit weakly aperiodic SFTs, that is, an SFT such that every configuration has an infinite orbit. Here we make use of the same technique to prove our result.

Before proceeding formally, let us first explain how we plan on using~\Cref{teorema_seward} to construct a grid-like structure on an arbitrary product of two infinite and finitely generated groups $H_1 \times H_2$. Fix an arbitrary finite set of generators $S_1$ for $H_1$ and consider the full $H_1$-shift on alphabet $S_1 \times S_1$. One may regard a configuration on $(S_1 \times S_1)^{H_1}$ as a list of labels on every element of $H_1$ indicating a left and right neighbour. Furthermore, restrict the set of allowable configurations by imposing the constraint that following the right neighbour and then, on the element just reached, following the left neighbour, one must end up in the initial element. We can interpret each configuration which satisfies this constraint as a coding of an action $\ZZ \curvearrowright H_1$. 

An action coded by a configuration as previously explained is bounded but not necessarily free. In fact, the orbit of each $h \in H_1$ may be either a finite cycle or an infinite copy of $\ZZ$. A priori, there may not exist any configurations such that each orbit is infinite. We shall use~\Cref{teorema_seward} to obtain that there exists at least one set of generators for which there exists a configuration which codes an action on which each orbit is free. From a geometric point of view, this means that there exists a finite set of generators $S_1$ for which the associated Cayley graph of $H_1$ can be partitioned in disjoint bi-infinite paths.

Proceed analogously on $H_2$. We have just obtained that there are configurations on $H_1$ and $H_2$ which represent partitions of the respective Cayley graphs on disjoint bi-infinite paths, see Figure~\ref{figure_grilla}. If we take the product of these two SFTs we obtain an $(H_1 \times H_2)$-SFT $\Grid$ for which each configuration represents an action $\ZZ^2 \curvearrowright H_1 \times H_2$ and there exists at least one of them which is free. In other words, the Cayley graph of $H_1 \times H_2$ is partitioned by copies of $\ZZ^2$.

\begin{figure}[h!]
	\centering
	\begin{tikzpicture}[scale=0.5]
	\begin{scope}[shift={(0,0)}]
	\node at (0,-4) {$H_1$};
	\clip[draw, decorate,decoration={random steps, segment length=2pt, amplitude=1pt}] (0,0) circle (3);
	\draw plot [smooth, tension=0.4] coordinates {(-3,4) (-0.5,2.5) (2,5)};
	\draw plot [smooth, tension=0.4] coordinates {(-3,3) (-0.5,2) (2,4)};
	\draw plot [smooth, tension=0.4] coordinates {(-3,2) (-1.5,2) (-0.4,1.5) (4,5)};
	\draw plot [smooth, tension=0.4] coordinates {(-3,1) (-2,1.5) (0.5,-0.5) (-0.4,1.2) (4,4)};
	\draw plot [smooth, tension=0.4] coordinates {(-3,0) (-2,1)(1,-1) (0,1) (4,4)};
	\draw  plot [smooth, tension=0.4] coordinates {(-3,-1) (-1.5,0)(1.4,-1.3) (0.5,0.5) (4,3)};
	\draw plot [smooth, tension=0.4] coordinates {(-3,-2) (-1.5,-0.7) (-1.3,-1.5) (-2,-1.8) (-1.8, -2.2) (-0.8,-2) (-0.5,-0.8) (1.6,-1.8) (2,-1.6) (1.5,0) (4,2)};
	\draw [very thick, black] plot [smooth, tension=0.4] coordinates {(-3,-3)  (-0.4,-2.2) (-0,-1.6) (1.6,-2.2) (2.2,-1.8) (1.9,-0.5) (4,1)};
	\draw plot [smooth, tension=0.4] coordinates {(-3,-3.5) (-0.4,-2.5) (2,-3)};
	\end{scope}
	\begin{scope}[shift={(8,0)}]
	\node at (0,-4) {$H_2$};
	\clip[draw, decorate,decoration={random steps, segment length=2pt, amplitude=1pt}] (0,0) circle (3);
	\draw plot [smooth, tension=0.4] coordinates {(2.5,-3) (2.5,-0.5) (3.5,3)};
	\draw plot [smooth, tension=0.4] coordinates {(2,-3) (2,-1) (3,3)};
	\draw plot [smooth, tension=0.4] coordinates {(1.5,-3) (1.6,0) (2.5,3)};
	\draw plot [smooth, tension=0.4] coordinates {(1,-3) (1,-0.5) (1.4,0.8) (2,3)};
	\draw plot [smooth, tension=0.4] coordinates {(0.5,-3) (0.4,-0.5) (1,0.5) (1,3)};
	\draw [very thick, black] plot [smooth, tension=0.4] coordinates {(0,-3) (-0.2,-0.5) (0.2,0.5) (0,3)};
	\draw plot [smooth, tension=0.4] coordinates {(-0.5,-3) (-2,-0.6) (-1.4,-0.4) (-0.4,-2.5) (-0.7,0) (-0.2,1) (-0.5,3)};
	\draw plot [smooth, tension=0.4] coordinates {(-1,-3) (-2.6,-0) (-1.5,0.4) (-0.6,1) (-1,3)};
	\draw plot [smooth, tension=0.4] coordinates {(-1.5,-3) (-3,0.5) (-2,0.8) (-1.2,1.2) (-2,3)};
	\draw plot [smooth, tension=0.4] coordinates {(-3,1.5) (-2,1.5) (-3,3)};
	\end{scope}
	\begin{scope}[shift={(15,-0.5)}, scale = 0.6]
	\draw [name path=l1,<->,thick, shift={(-0.1,0)}] plot [smooth, tension=0.4] coordinates {(-3,-3)  (-0.4,-2.2) (-0,-1.6) (1.6,-2.2) (2.2,-1.8) (1.9,-0.5) (4,1)};
	\draw [name path=l2,<->,thick, shift={(-0.1,2)}] plot [smooth, tension=0.4] coordinates {(-3,-3)  (-0.4,-2.2) (-0,-1.6) (1.6,-2.2) (2.2,-1.8) (1.9,-0.5) (4,1)};
	\draw [name path=l3,<->,thick, shift={(-0.1,4)}] plot [smooth, tension=0.4] coordinates {(-3,-3)  (-0.4,-2.2) (-0,-1.6) (1.6,-2.2) (2.2,-1.8) (1.9,-0.5) (4,1)};
	\draw [name path=la,<->,thick, shift={(-2,0)}] plot [smooth, tension=0.4] coordinates {(0,-3.5) (-0.2,-0.5) (0.2,0.5) (0,3) (0,5)};
	\draw [name path=lb,<->,thick, shift={(-0.5,0)}] plot [smooth, tension=0.4] coordinates {(0,-3.5) (-0.2,-0.5) (0.2,0.5) (0,3) (0,5)};
	\draw [name path=lc,<->,thick, shift={(1,0)}] plot [smooth, tension=0.4] coordinates {(0,-3.5) (-0.2,-0.5) (0.2,0.5) (0,3) (0,5)};
	\draw [name path=ld,<->,thick, shift={(2.5,0)}] plot [smooth, tension=0.4] coordinates {(0,-3.5) (-0.2,-0.5) (0.2,0.5) (0,3) (0,5)};
	\draw[name intersections={of=l1 and la,total=\t}, thick, fill = black!50]
	{(intersection-1) circle (5pt)};
	\draw[name intersections={of=l1 and lb,total=\t}, thick, fill = black!50]
	{(intersection-1) circle (5pt)};
	\draw[name intersections={of=l1 and lc,total=\t}, thick, fill = black!50]
	{(intersection-1) circle (5pt)};
	\draw[name intersections={of=l1 and ld,total=\t}, thick, fill = black!50]
	{(intersection-1) circle (5pt)};
	
	\draw[name intersections={of=l2 and la,total=\t}, thick, fill = black!50]
	{(intersection-1) circle (5pt)};
	\draw[name intersections={of=l2 and lb,total=\t}, thick, fill = black!50]
	{(intersection-1) circle (5pt)};
	\draw[name intersections={of=l2 and lc,total=\t}, thick, fill = black!50]
	{(intersection-1) circle (5pt)};
	\draw[name intersections={of=l2 and ld,total=\t}, thick, fill = black!50]
	{(intersection-1) circle (5pt)};
	
	\draw[name intersections={of=l3 and la,total=\t}, thick, fill = black!50]
	{(intersection-1) circle (5pt)};
	\draw[name intersections={of=l3 and lb,total=\t}, thick, fill = black!50]
	{(intersection-1) circle (5pt)};
	\draw[name intersections={of=l3 and lc,total=\t}, thick, fill = black!50]
	{(intersection-1) circle (5pt)};
	\draw[name intersections={of=l3 and ld,total=\t}, thick, fill = black!50]
	{(intersection-1) circle (5pt)};
	
	\end{scope}
	\end{tikzpicture}
	\caption{Finding a grid in $H_1 \times H_2$}
	\label{figure_grilla}
\end{figure}
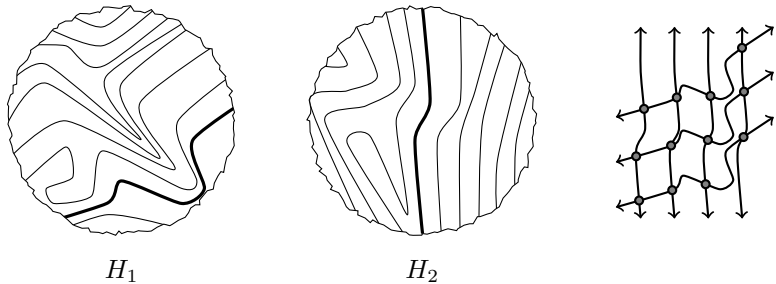

Finally, we shall use the structure of $\Grid$ to embed a nearest neighbour $\ZZ^2$-SFT $Y$.

Let us now proceed formally. For $i \in \{1,2\}$ consider a finite and symmetric set of generators $S_i$ for $H_i$ which contains the identity. Let $\mathcal{G}_i = S_i \times S_i$ be the alphabet of all pairs of generators. For $\mathsf{s} = (s_1,s_2) \in \mathcal{G}_i$ denote 

\begin{align*}
\texttt{left}_i(\mathsf{s}) = s_1, & &\texttt{right}_i(\mathsf{s}) = s_2.
\end{align*}

We may think of $\texttt{left}_i$ and $\texttt{right}_i$ as labels on each $h \in H_i$ pointing towards a left and right neighbour. Let $\Grid_i \subset \mathcal{G}_i^{H_i}$ be the subshift defined by forbidding all patterns $p$ with support on $\{1_{H_i},s\}$ for some $s \in S_i$ and such that either
\begin{itemize}
	\item $\texttt{right}_i(p(1_{H_i})) = s$ but $\texttt{left}_i(p(s)) \neq s^{-1}$.
	\item $\texttt{left}_i(p(1_{H_i})) = s$ but $\texttt{right}_i(p(s)) \neq s^{-1}$.
\end{itemize}

There are at most finitely many such patterns, and thus $\Grid_i$ is an $H_i$-SFT. Because of this constraint, each $\omega_i \in \Grid_i$ defines an action $[\omega_i]: \ZZ \curvearrowright H_i$ as follows:
\begin{align*}
[\omega_i](1,h) \isdef h\cdot \texttt{right}_i(\omega_i(h)), & &[\omega_i](-1,h) \isdef h \cdot \texttt{left}_i(\omega_i(h)).
\end{align*}

The forbidden patterns ensure that this is a $\ZZ$-action, more precisely, we have that $[\omega_i](-1,[\omega_i](1,h)) = [\omega_i](1,[\omega_i](-1,h)) = h$.

\begin{proposition}\label{proposition_ofseward}
	There exists a finite and symmetric set of generators $S_i$ for $H_i$  which contains the identity $1_{H_i}$ and such that the subshift $\Grid_i$ defined using that set of generators contains a configuration $\omega_i$ which acts freely on $H_i$.
\end{proposition}

\begin{proof}
	By \Cref{teorema_seward} there exists a translation-like action $f_i\colon \ZZ \curvearrowright H_i$. Fix a preliminary set of generators $\bar{S}_i$ of $H_i$, consider the associated generator metric $\bar{d}_i$ on $H_i$ and define $$S_i \isdef \left\{ s \in H_i \mid \bar{d}_i(1_{H_i},s) \leq \sup_{h \in H_i} \bar{d}_i(h, f_i(1,h)) \right\}.$$ 
	As $f_i$ is a translation-like, we know that the distance from $f_i(1,h)$ to $h$ is uniformly bounded, therefore $S_i$ is finite. Furthermore, $S_i$ contains $1_{H_i}$ and is symmetric. We claim that $S_i$ satisfies the requirements. Indeed, we can define $\omega_i \in \Grid_i$ by:
	$$\omega_i(h) \isdef (h^{-1}f_i(-1,h), h^{-1}f_i(1,h)).$$
	As $f_i$ is uniformly bounded and by definition of $S_i$ both $h^{-1}f_i(-1,h)$ and $h^{-1}f_i(1,h)$ are in $S_i$. Furthermore, we have \begin{align*}
	[\omega_i](1,h) &= h \cdot \texttt{right}(\omega_i(h))= h \cdot (h^{-1}f_i(1,h)) = f_i(1,h)\\
	[\omega_i](-1,h) &= h \cdot \texttt{left}(\omega_i(h))= h \cdot (h^{-1}f_i(-1,h)) = f_i(-1,h)
	\end{align*}
	and thus the action $[\omega_i]$ induced by $\omega_i$ is the same as $f_i$, which is free.
\end{proof}

For the rest of the section, let $S_1,S_2$ be sets of generators for $H_1$ and $H_2$ respectively which satisfy the conditions of~\Cref{proposition_ofseward}. We can extend the $H_1$-SFT $\Grid_1$ to an $(H_1\times H_2)$-SFT $\overline{\Grid}_1$ by imposing that $H_2$ acts trivially. That is, $\overline{\omega} \in \overline{\Grid}_1$ if and only if $\overline{\omega}(h_1,h_2) = \overline{\omega}(h_1, 1_{H_2})$ for every $h_2 \in H_2$ and $\overline{\omega}|_{H_1 \times \{1_{H_2}\}} \in \Grid_1$. Analogously, we extend $\Grid_2$ to an $(H_1\times H_2)$-SFT $\overline{\Grid}_2$ by imposing that $H_1$ acts trivially. Finally, we define $\Grid \isdef \overline{\Grid}_1 \times \overline{\Grid}_2$. By definition, if $\omega = (\overline{\omega}_1,\overline{\omega}_2)$ we may naturally associate configurations $\omega_1 \in \Grid_1$ and $\omega_2 \in \Grid_2$ such that $$\omega(h_1,h_2) = (\overline{\omega}_1(h_1,h_2),\overline{\omega}_2(h_1,h_2)) = (\omega_1(h_1),\omega_2(h_2)).$$

In particular, each $\omega \in \Grid$ induces a $\ZZ^2$-action $[\omega] = [\omega_1]\times [\omega_2]$ on $H_1 \times H_2$, and by~\Cref{proposition_ofseward} there is one $\omega$ such that the action is free.

Finally, we shall use the subshift $\Grid$ to embed nearest neighbour $\ZZ^2$-subshifts. Let $Y \subset (\ag_Y)^{\ZZ^2}$ be a $\ZZ^2$-SFT given by a nearest neighbour set of forbidden patterns $\FF_Y$ over some finite alphabet $\ag_Y$. We define the $(H_1\times H_2)$-SFT $\Grid(Y)$ over the alphabet $\mathcal{G}_1 \times \mathcal{G}_2 \times \ag_Y$ as the subset of configurations in $\Grid \times \ag_Y^{H_1 \times H_2}$ which is further constrained by an additional set of forbidden patterns $\FF_{\Grid(Y)}$.

Before formally defining $\FF_{\Grid(Y)}$, let us explain the geometric meaning of $\Grid(Y)$. If we denote an element of $\Grid(Y)$ as a pair $(\omega,y) \in \Grid \times \ag_Y^{H_1 \times H_2}$. We have that $\omega$ is an element of $\Grid$ and as such indicates the neighbours on the simulated $\ZZ^2$-grid. The second component $y$ is adding a symbol from $\ag_Y$ to each element of $H_1 \times H_2$. We shall define $\FF_{\Grid(Y)}$ in such a way that reading the symbols in $y$ in a simulated $\ZZ^2$ grid according to $\omega$ yields a valid configuration in $Y$. Let us remark that a pattern $p$ over the alphabet $\mathcal{G}_1 \times \mathcal{G}_2 \times \ag_Y$ can be seen as a pair $(p_1,p_2) \in (\mathcal{G}_1 \times \mathcal{G}_2)^{\supp(p)} \times (\ag_Y)^{\supp(p)}$.

Let us define $\FF_{\Grid(Y)}$ as the set of patterns $p = (p_1,p_2)$ with support $\supp(p) = S_1 \times S_2$ such that the following condition holds:

Define $q_{p,\texttt{hor}} \in (\ag_Y)^{\{(0,0),(1,0)\}}$ and $q_{p,\texttt{ver}} \in (\ag_Y)^{\{(0,0),(0,1)\}}$ by setting:
\begin{align*}
q_{p,\texttt{hor}}(0,0)= p_2(1_{H_1},1_{H_2}) &  & q_{p,\texttt{hor}}(1,0) = p_2(\texttt{right}_1(p_1(1_{H_1},1_{H_2})),1_{H_2})\\
q_{p,\texttt{ver}}(0,0)= p_2(1_{H_1},1_{H_2}) &  & q_{p,\texttt{ver}}(0,1) = p_2(1_{H_1},\texttt{right}_2(p_1(1_{H_1},1_{H_2})))
\end{align*}

Declare $p \in \FF_{\Grid(Y)}$ if and only if $q_{p,\texttt{hor}} \in \FF_Y$ or $q_{p,\texttt{ver}} \in \FF_Y$. In other words, we forbid patterns where forbidden patterns from $\FF_Y$ appear along a neighbour defined by the $\Grid$ component.

\begin{remark}\label{remark_galletasfrak}
	Given a configuration $(\omega,y) \in \Grid(Y)$ and a pair $(h_1,h_2)\in H_1 \times H_2$ we can extract a configuration of $Y$ by following the action $[\omega] = [\omega_1]\times [\omega_2]$, namely, we can define a function $\mathfrak{C}\colon \Grid(Y) \times H_1 \times H_2 \to (\ag_Y)^{\ZZ^2}$ by setting for each $(z_1,z_2) \in \ZZ^2$
	\begin{align*}
	\mathfrak{C}((\omega,y),h_1,h_2)(z_1,z_2) & \isdef y([\omega]((z_1,z_2),(h_1,h_2)))\\
	& = y([\omega_1](z_1,h_1),[\omega_2](z_2,h_2)).
	\end{align*}
	
	In simpler words, $\mathfrak{C}((\omega,y),h_1,h_2)$ is the configuration obtained by reading $y$ along the two-dimensional grid defined by $[\omega]$ where the pair $(h_1,h_2)$ is identified to $(0,0)$. By definition of $\FF_{\Grid(Y)}$, we have that $\mathfrak{C}((\omega,y),h_1,h_2) \in Y$ as no forbidden patterns from $\FF_Y$ can occur. 
\end{remark}

\begin{remark}
	If the action $[\omega]$ is not free there exist $(h_1,h_2)$ and $(z_1,z_2) \in \ZZ^2 \setminus \{(0,0)\}$ such that $[\omega]((z_1,z_2),(h_1,h_2)) = (h_1,h_2)$ and this implies that $\mathfrak{C}((\omega,y),h_1,h_2){(z_1,z_2)} =\mathfrak{C}((\omega,y),h_1,h_2){(0,0)}$. In particular, if $Y$ is a nearest neighbour strongly aperiodic $\ZZ^2$-subshift, for instance, a nearest neighbour recoding of the Robinson tiling~\cite{Robinson1971} we have that every pair $(\omega,y) \in \Grid(Y)$ must satisfy that $[\omega]$ is free.
\end{remark}

Potentially, the subshift $\Grid(Y)$ may be empty. In the next proposition we show that this is not the case.

\begin{proposition}\label{propo_relacion}
	Let $Y$ be a nearest neighbour $\ZZ^2$-subshift. For each $c \in Y$ there exists $(\omega,y) \in \Grid(Y)$ such that for every $(h_1,h_2) \in H_1 \times H_2$ we have that $\mathfrak{C}((\omega,y),h_1,h_2) \in \Orb_{\sigma}(c)$ and $c = \mathfrak{C}((\omega,y),1_{H_1},1_{H_2})$.
\end{proposition}

\begin{proof}
	Let $\omega$ be an element of $\Grid$ such that $[\omega]$ is free. We begin by fixing an arbitrary starting point in every simulated grid, formally, define the equivalence relation over $H_1 \times H_2$ where $(h_1,h_2)\sim(h'_1,h'_2)$ if and only if there is $(z_1,z_2) \in \ZZ^2$ such that $[\omega]((z_1,z_2),(h_1,h_2)) = (h'_1,h'_2)$. Let $(h^i_1,h^i_2)_{i \in I}$ be a representing set of $(H_1 \times H_2)/\sim$ which contains $(1_{H_1},1_{H_2})$. We define $y \in \ag_Y^{H_1 \times H_2}$ by $$y({[\omega]((z_1,z_2),(h^i_1,h^i_2))}) \isdef c(z_1,z_2).$$
	By definition of $\sim$ and freeness of $[\omega]$, $y$ is well defined over all of $H_1 \times H_2$. Moreover, by definition of $\FF_{\Grid(Y)}$, we have that $(\omega,y) \in \Grid(Y)$. 
	
	Let $(z_1,z_2) \in \ZZ^2$ and $i \in I$ such that $(h_1,h_2) = [\omega]((z_1,z_2),(h^i_1,h^i_2))$. We have $\mathfrak{C}((\omega,y),h_1,h_2)(z'_1,z'_2) = c({(z_1,z_2)+(z'_1,z'_2)})$ and thus $\mathfrak{C}((\omega,y),h_1,h_2) = \sigma^{-(z_1,z_2)}(c) \in \Orb_{\sigma}(c)$. Finally, as $(1_{H_1},1_{H_2})$ belongs to the representing set we obtain $\mathfrak{C}((\omega,y),1_{H_1},1_{H_2}) = \sigma^{(0,0)}(c) = c.$
\end{proof}

\subsection{Proof of \Cref{simulationTHEOREM}}\label{subsecproof}

The final part of the proof of~\Cref{simulationTHEOREM} is quite simple, albeit a little heavy on notation. First, we extract a nearest neighbour extension $Y$ of $\Topzz(X,T)$ and embed it as an $(H_1 \times H_2)$-SFT using $\Grid(Y)$ as defined in~\Cref{subsecgrid}. Second, we construct a subshift $\Final(X,Y)$ over the group $G \times H_1 \times H_2$ which contains in every $(H_1 \times H_2)$-coset a copy of $\Grid(Y)$. We can thus think of a configuration in $\Final(X,T)$ as a sequence $(\omega_g,y_g)_{g \in G}$ of configurations in $\Grid(Y)$. We shall impose through local rules that in every such configuration all the $\omega_g$ are the same, that is, that every $(H_1 \times H_2)$-coset has the same grid structure. Furthermore, we shall also impose through local rules that if $(y_g)$ is coding a configuration $x \in X$, then for every generator $s \in S$ of $G$ we have that $y_{s^{-1}g}$ codes $T^s(x)$. This is also done through local rules. Finally, we shall define our factor map using the function $\gamma$ defined on~\Cref{lemma_allforone} and prove that it satisfies the required properties.

Let us now proceed formally. Consider first the subshift $\Topzz(X,T)$ from \Cref{lemma_allforone}. By~\Cref{remark_nn_and_1_block} we may extract a nearest neighbour $\ZZ^2$-SFT extension $Y$ and a $1$-block map $\widehat{\phi}\colon Y \twoheadrightarrow \Topzz(X,T)$ defined by a local function $\widehat{\Phi}$. We can construct the non-empty $(H_1 \times H_2)$-SFT $\Grid(Y)$ as defined in~\Cref{subsecgrid}. Let $\FF_{\Grid(Y)}$ be a finite set of forbidden patterns defining $\Grid(Y)$. 

Let $S$ be the finite set of generators of $G$ with which  $\Topz(X,T)$ was defined. We construct a $(G \times H_1 \times H_2)$-subshift $\Final(X,T)$ using the alphabet of $\Grid(Y)$ and the set of forbidden patterns $\FF_{\Final(X,T)} \isdef \FF_1 \cup \FF_2 \cup \FF_3$. Before defining these sets formally let us describe them in simpler words.

$\FF_1$ will be a set of forbidden patterns which forces each $(H_1 \times H_2)$-coset to contain configurations only from $\Grid(Y)$. $\FF_2$ will further constraint $\Final(X,T)$ in such a way that in every $(H_1 \times H_2)$-coset the first component (the $\omega \in \Grid$) is the same. Finally, $\FF_3$ will link the different $(H_1 \times H_2)$-cosets indexed by $G$ forcing them to respect the action of $T$. The effect of this last set of forbidden patterns is illustrated on~\Cref{figure_reglas} where the grids are induced by $\omega$ and appear somewhere in $\{1_G\} \times H_1 \times H_2$ and $\{s_1^{-1}\} \times H_1 \times H_2$. In each of the two grids a configuration from a nearest neighbour extension of $\Topzz(X,T)$ is encoded. The forbidden patterns from $\FF_3$ code the fact that if the configurations indexed by $1_G$ and $s_1^{-1}$ code respectively $x$ and $y$, we will forcefully have that $y = T^{s_1}(x)$.

\begin{figure}[h!]
	\centering
	\begin{tikzpicture}[scale=0.5]
	\node at (-7,0) {$1_G$};
	\node at (-7,5) {$s_1^{-1}$};
	\node at (9,5) {{\footnotesize $\begin{pmatrix}
			\Psi(T^{s_1}(x)) \\
			\Psi(T^{s_1s_1}(x)) \\
			\vdots \\
			\Psi(T^{s_ns_1}(x)) 
			\end{pmatrix}$}};
	\draw[color = black, dashed, fill = black!10] (13,6) rectangle (15,6.75);
	\draw[color = black, dashed, fill = black!10] (12.5,0.25) rectangle (15.5,1);
	\node at (11.7,5) {$=$};
	\node at (11.7,5.4) {{\tiny$\FF_3$}};
	\node at (14,0) {{\footnotesize $\begin{pmatrix}
			\Psi(x) \\
			\Psi(T^{s_1}(x)) \\
			\vdots \\
			\Psi(T^{s_n}(x)) 
			\end{pmatrix}$}};
	\node at (14,5) {{\footnotesize $\begin{pmatrix}
			\Psi(y) \\
			\Psi(T^{s_1}(y)) \\
			\vdots \\
			\Psi(T^{s_n}(y)) 
			\end{pmatrix}$}};
	\begin{scope}[shift={(3,0)},scale = 0.8]
	\def \c{0.3}
	\node (aa) at (0,0) {};
	\draw [<->] (-5,0) to (5+5/3,0);
	\draw [<->] (-6,-1) to (4+5/3,-1);
	\draw [<->] (-7,-2) to (3+5/3,-2);
	\draw [<->] (-4,1) to (6+5/3,1);
	\draw [<->] (-3,2) to (7+5/3,2);

	\draw [<->] (3,3) to (-3,-3);
	\draw [<->] (3-5/3,3) to (-3-5/3,-3);
	\draw [<->] (3-10/3,3) to (-3-10/3,-3);
	\draw [<->] (3+5/3,3) to (-3+5/3,-3);
	\draw [<->] (3+10/3,3) to (-3+10/3,-3);
	\draw [<->] (3+15/3,3) to (-3+15/3,-3);
	
	\foreach \i in {-10/3,-5/3,0,5/3,10/3,15/3}{
		\foreach \j in {-2,-1,0,1,2}{
			\draw[fill = black!20] (\i+\j,\j) circle (\c);
		}
	}
	\end{scope}
	
	\begin{scope}[shift={(0,5)},scale = 0.8]
	\def \c{0.3}
	\node (bb) at (0,0) {};
	\draw [<->] (-5,0) to (5+5/3,0);
	\draw [<->] (-6,-1) to (4+5/3,-1);
	\draw [<->] (-7,-2) to (3+5/3,-2);
	\draw [<->] (-4,1) to (6+5/3,1);
	\draw [<->] (-3,2) to (7+5/3,2);

	\draw [<->] (3,3) to (-3,-3);
	\draw [<->] (3-5/3,3) to (-3-5/3,-3);
	\draw [<->] (3-10/3,3) to (-3-10/3,-3);
	\draw [<->] (3+5/3,3) to (-3+5/3,-3);
	\draw [<->] (3+10/3,3) to (-3+10/3,-3);
	\draw [<->] (3+15/3,3) to (-3+15/3,-3);
	
	\foreach \i in {-10/3,-5/3,0,5/3,10/3,15/3}{
		\foreach \j in {-2,-1,0,1,2}{
			\draw[fill = black!20] (\i+\j,\j) circle (\c);
		}
	}
	\end{scope}
	\draw [->, bend left] (aa) to (bb);
	\node at (0,2) {{\footnotesize$s_1^{-1}$}};
	\end{tikzpicture}
	\caption{The set $\FF_3$ links the different $(H_1\times H_2)$-cosets by forcing that in all configurations as above then $\Psi(y) = \Psi(T^{s_1}(x))$.}
	\label{figure_reglas}
\end{figure}
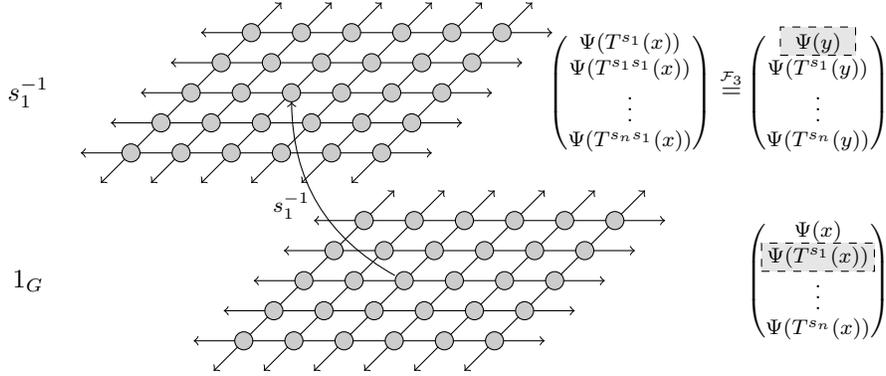

Let us now define the set of forbidden patterns $\FF_{\Final(X,T)}  \isdef \FF_1 \cup \FF_2 \cup \FF_3$.

\begin{enumerate}
	\item For each $q \in \FF_{\Grid(Y)}$ let $p$ be the pattern with $\supp(p) = \{1_G\} \times \supp(q)$ such that for every $(h_1,h_2) \in \supp(q)$ we have $p{(1_G,h_1,h_2)} = q{(h_1,h_2)}$. We define $\FF_1$ as the set of all such $p$.
	\item We define $\FF_2$ as the set of patterns $p = (p_1,p_2)$ which have support $\supp(p) =$ $\{(1_G,1_{H_1},1_{H_2}), (s,1_{H_1},1_{H_2})\}$ for some $s \in S$ and such that:
	$$p_1(1_G,1_{H_1},1_{H_2}) \neq p_1(s,1_{H_1},1_{H_2})$$
	\item We define $\FF_3$ as the set of patterns $p=(p_1,p_2)$ which have support $\supp(p) =$  $\{(1_G,1_{H_1},1_{H_2}), (s^{-1},1_{H_1},1_{H_2})\}$ for some $s \in S$ and such that: $$\widehat{\Phi}(p_2(1_G,1_{H_1},1_{H_2}))(s) \neq \widehat{\Phi}(p_2(s^{-1},1_{H_1},1_{H_2}))(1_G).$$
\end{enumerate}

As the union of the supports of the patterns in $\FF_{\Final(X,T)} $ is bounded, we have that $\Final(X,T)$ is an SFT. We shall denote elements of $\Final(X,T)$ as pairs $(\omega,y)$ and their restrictions to the set $\{1_G\}\times H_1 \times H_2$ by $(\omega^0,y^0) \isdef (\omega,y)|_{\{1_G\}\times H_1 \times H_2}$. From $\FF_1$ we obtain that $(\omega^0,y^0) \in \Grid(Y)$. Furthermore, from the definition of $\FF_2$ we deduce that $\omega$ must satisfy $\omega(g_1,h_1,h_2) = \omega(1_G,h_1,h_2)$ for every $g \in G$, therefore we may identify $\omega$ with $\omega^0 \in \Grid$ and thus unambiguously define an action $[\omega] : \ZZ^2 \curvearrowright H_1 \times H_2$.

To prove~\Cref{simulationTHEOREM} it suffices to show that its $G$-subaction is an extension of $(X,T)$. Indeed, consider the map $\varphi \colon \Final(X,T) \to X$ defined as follows: for $(\omega,y) \in \Final(X,T)$ let $(\omega^0,y^0) \isdef (\omega,y)|_{\{1_G\}\times H_1 \times H_2}$ be defined as above. Using the notation from~\Cref{lemma_allforone} and~\Cref{remark_galletasfrak} we define $$\varphi(\omega,y) \isdef \gamma\left(1_G,\widehat{\phi}\left(\mathfrak{C}((\omega^0,y^0),1_{H_1},1_{H_2})\right)\right).$$

By the definition of $\FF_1$, we have that $(\omega^0,y^0) \in \Grid(Y)$ and hence we obtain that $\mathfrak{C}((\omega^0,y^0),1_{H_1},1_{H_2}) \in Y$ as shown in~\Cref{propo_relacion}. In turn, $\widehat{\phi}(\mathfrak{C}((\omega^0,y^0),1_{H_1},1_{H_2})) \in \Topzz(X,T)$ and thus we obtain $\gamma(1_G, \widehat{\phi}(\mathfrak{C}((\omega^0,y^0),1_{H_1},1_{H_2}))) \in X$ and so $\varphi(\omega,y) \in X$. Moreover, as $\widehat{\phi}$ is a $1$-block map, in order to compute the first $n$ coordinates of $\varphi(\omega,y)$ it suffices to know the values of $\mathfrak{C}((\omega',y'),1_{H_1},1_{H_2})$ restricted to a ball of diameter $3^n$ of $\ZZ^2$. And in turn, it suffices to know $(\omega',y')$ restricted to the ball of diameter $3^n$ of $H_1 \times H_2$ with respect to the generators $S_1 \times S_2$. This means that $\varphi(\omega,y)$ is continuous. In order to conclude we need to show that $\varphi$ is onto and that it is $G$-equivariant.

\begin{claim}
	$\varphi$ is $G$-equivariant.
\end{claim}

\begin{proof}
	We need to show that for every $(\omega,y) \in \Final(X,T)$ and $g \in G$ we have $\varphi(\sigma^g(\omega,y)) = T^g(\varphi(\omega,y))$. Clearly, it suffices to show the property for each $s \in S$. Let $(\omega,y) \in \Final(X,T)$ and denote $(\omega^0,y^0) \isdef (\omega,y)|_{\{1_G\}\times H_1 \times H_2}$ and $(\omega^1,y^1) \isdef \sigma^{s}(\omega,y)|_{\{1_{G}\}\times H_1 \times H_2}$. As no patterns from $\FF_2$ appear in $(\omega,y)$, we have that $\omega^0 = \omega^1$ and thus their induced actions are the same, that is $[\omega^0] = [\omega^1]$ and thus we may denote both actions just by $[\omega]$. Using this we get that on the one hand, for each $(z_1,z_2) \in \ZZ^2$ \begin{align*}
	\widehat{\phi}(\mathfrak{C}((\omega^1,y^1),1_{H_1},1_{H_2}))(z_1,z_2) &  = \widehat{\Phi}(\mathfrak{C}((\omega^1,y^1),1_{H_1},1_{H_2})(z_1,z_2))\\
	& = \widehat{\Phi}(y^1([\omega]((z_1,z_2),(1_{H_1},1_{H_2}))))\\
	& = \widehat{\Phi}(y(s^{-1},[\omega]((z_1,z_2),(1_{H_1},1_{H_2})))).
	\end{align*}
	And on the other hand,
	\begin{align*}
	\widehat{\phi}(\mathfrak{C}((\omega^0,y^0),1_{H_1},1_{H_2}))(z_1,z_2) &  = \widehat{\Phi}(\mathfrak{C}((\omega^0,y^0),1_{H_1},1_{H_2})(z_1,z_2))\\
	& = \widehat{\Phi}(y^0([\omega]((z_1,z_2),(1_{H_1},1_{H_2}))))\\
	& = \widehat{\Phi}(y(1_G,[\omega]((z_1,z_2),(1_{H_1},1_{H_2})))). \\
	\end{align*}
	Putting the previous equations together with the fact that no patterns from $\FF_3$ appear in $(\omega,y)$, we obtain $$\widehat{\phi}(\mathfrak{C}((\omega^1,y^1),1_{H_1},1_{H_2}))(1_G) = \widehat{\phi}(\mathfrak{C}((\omega^0,y^0),1_{H_1},1_{H_2}))(s).$$ Finally, a simple computation yields: \begin{align*}
	\varphi(\sigma^s(\omega,y)) & = \gamma(1_G, \widehat{\phi}(\mathfrak{C}((\omega^1,y^1),1_{H_1},1_{H_2})))\\
	& = \gamma(s, \widehat{\phi}(\mathfrak{C}((\omega^0,y^0),1_{H_1},1_{H_2})))\\
	& = T^s(\gamma(1_G,\widehat{\phi}(\mathfrak{C}((\omega^0,y^0),1_{H_1},1_{H_2})))) \\
	& = T^s(\varphi(\omega,y)).
	\end{align*}
	Where the penultimate equality is from \Cref{lemma_allforone}.
\end{proof}

\begin{claim}\label{varphi_onto}
	The map $\varphi$ is onto.
\end{claim}

\begin{proof}
	Let $x \in X$. For each $g \in G$ let $c^g \in Y$ be a preimage under $\widehat{\phi}$ of the vertical extension of $\widetilde{\Psi}(T^g(x)) \in \Topz(X,T)$ as defined in~\Cref{subsectop}. Also, choose $\bar{\omega} \in \Grid$ such that $[\bar{\omega}]$ is free. By~\Cref{propo_relacion}, we obtain that for each $g \in G$ there exists a pair $(\bar{\omega},y^g) \in \Grid(Y)$ such that $\mathfrak{C}((\bar{\omega},y^g),1_{H_1},1_{H_2}) = c^g$. Moreover, by fixing a set of representatives $(h^i_1,h^i_2)_{i \in I}$ of $(H_1 \times H_2)/\sim$ as in \Cref{propo_relacion}, we can impose that for each $(h_1,h_2) \in H_1 \times H_2$ we have $\mathfrak{C}((\bar{\omega},y^g),h_1,h_2) = \sigma^{-(z_1,z_2)}(c^g)$ for the unique $(z_1,z_2) \in \ZZ^2$ such that there is an $i \in I$ satisfying $[\bar{\omega}]((z_1,z_2),(h^i_1,h^i_2)) = (h_1,h_2)$. We define the $G \times H_1 \times H_2$ configuration $(\omega,y)$ as follows:
	$$(\omega,y)(g,h_1,h_2) \isdef (\bar{\omega}(h_1,h_2), y^{g^{-1}}(h_1,h_2)).$$
	By definition, we have \begin{align*}
	\varphi(\omega,y) & = \gamma(1_G,\widehat{\phi}(\mathfrak{C}((\bar{\omega},y^{1_G}),1_{H_1},1_{H_2})))\\
	& = \gamma(1_G,\widehat{\phi}(c^{1_G})) \\
	& = \gamma(1_G,\widetilde{\Psi}(x)) \\
	& = x.
	\end{align*}
	Therefore, it suffices to show that $(\omega,y) \in \Final(X,T)$. As every $(H_1 \times H_2)$-coset contains a configuration from $\Grid(Y)$, no patterns from $\FF_1$ appear. Also, as the first component is always $\bar{\omega}$, we have that no patterns from $\FF_2$ appear. Finally, we have that for every $g \in G$ and $s \in S$ then $y{(g^{-1},h_1,h_2)} = (y^{g}){(h_1,h_2)} = (c^g)(z_1,z_2)$ and $y{(g^{-1}s^{-1},h_1,h_2)} = (y^{sg}){(h_1,h_2)} = (c^{sg})(z_1,z_2)$.
	
	Therefore we have that, for each $s \in S$, $\widehat{\Phi}((c^g)(z_1,z_2))(s) = (\widetilde{\Psi}(T^g(x))(z_1))(s)$ and $\widehat{\Phi}((c^{sg})(z_1,z_2))(s) = (\widetilde{\Psi}(T^{sg}(x))(s))(z_1)$. In particular as $$\widetilde{\Psi}(T^g(x))(s) = {\Psi}(T^s(T^g(x))) = {\Psi}(T^{sg}(x)) = \widetilde{\Psi}(T^{sg}(x))(1_G)$$
	we get that $\widehat{\Phi}((c^g)(z_1,z_2))(s) = \widehat{\Phi}((c^{sg})(z_1,z_2))(1_G)$ and thus, $$\widehat{\Phi}(y{(g^{-1},h_1,h_2)})(s) = \widehat{\Phi}(y{(g^{-1}s^{-1},h_1,h_2)})(1_G).$$ This implies that no patterns from $\FF_3$ appear. Therefore $(\omega,y) \in \Final(X,T)$.
\end{proof}

Collecting both claims and the previously proven properties of $\varphi$, we conclude that $\varphi \colon (\Final(X,T),\sigma_G) \twoheadrightarrow (X,T)$ is a topological factor. This proves~\Cref{simulationTHEOREM}.

\section{Consequences and remarks}\label{section.consequences}

In this last section we explore some consequences of \Cref{simulationTHEOREM}. In the case of expansive actions, we can give more detailed information about the factor. Specifically, we show that if $G$ is a recursively presented group, then every effectively closed $G$-subshift can be realized as the projective subdynamics of a sofic $(G \times H_1 \times H_2)$-subshift. Moreover, we prove that the sofic subshift can be picked in such a way that it is invariant under the shift action of $H_1 \times H_2$. 

This result is particularly helpful for the next part where we show that any group that can be written as the direct product of three infinite and finitely generated groups with decidable word problem admits a non-empty strongly aperiodic SFT.

Finally, we close this section by showing how the previous result can be used to prove the existence of non-empty strongly aperiodic subshifts in a class of branch groups which includes the Grigorchuk group.

\subsection{The case of effectively closed expansive actions}

The subshift $\Final(X,T)$ constructed in the proof of \Cref{simulationTHEOREM} satisfies the required properties, however, it has an undesirable perk. Namely, it might happen that for $(\omega,y) \in \Final(X,T)$ we have $\varphi(\omega,y) \neq \varphi(\sigma^{(1_G,h_1,h_2)}(\omega,y))$ for some $(h_1,h_2) \in H_1 \times H_2$. The reason is that in $\omega|_{\{1_G\}\times H_1 \times H_2} \in \Grid$ there might be many different coded grids and a priori there is no restriction forcing them to contain shifts of the same configuration.

The natural approach to get rid of this perk is to use the functions $\gamma_n$ defined after~\Cref{propo_effective_subshift} to impose in every $(H_1 \times H_2)$-coset that the first $n$-coordinates of the coded configuration are the same everywhere. While this works naturally for an expansive action, it fails in the case where $(X,T)$ is equicontinuous, see Proposition 6.1 of~\cite{Hochman2009b} for a simple example. This makes expansive systems particularly interesting in this construction, especially in the proof of \Cref{corollaryaperioci} where we show that every triple direct product of finitely generated groups with decidable word problem admit strongly aperiodic SFTs. 

So far, we have only used the notion of effectively closed subshift in the context of $\ZZ$-subshifts. In the case of a general finitely generated group, we need a way to code forbidden patterns into a language. This is achieved by the notion of pattern coding, for a longer survey see~\cite{ABS2017}.

Given a group $G$ generated by a finite set $S$ and a finite alphabet $\ag$ a \define{pattern coding} $c$ is a finite set of tuples $c=(w_i,a_i)_{i \in I}$ where $w_i \in S^{*}$ and $a_i \in \ag$. A set of pattern codings $\CC$ is said to be recursively enumerable if there is a Turing machine which takes as input a pattern coding $c$ and accepts it if and only if $c \in \CC$. We remark that every pattern can be coded by identifying each element in its support to some word in $S^*$ representing it.

\begin{definition}
	A subshift $X \subset \ag^G$ is \define{effectively closed} if there is a recursively enumerable set of pattern codings $\CC$ such that:$$ X = X_{\CC} \isdef \bigcap_{g \in G, c \in \CC} \left( \ag^G \setminus \bigcap_{(w,a) \in c}[a]_{gw} \right).$$
\end{definition}

\begin{theorem}\label{theorem_expansive_case}
	Let $G$ be a recursively presented and finitely generated group and $Y$ an effectively closed $G$-subshift. For every pair of infinite and finitely generated groups $H_1,H_2$ there exists a sofic $(G \times H_1 \times H_2)$-subshift $X$ such that:
	\begin{itemize}
		\item The $G$-subaction of $X$ is conjugate to $Y$.
		\item The $G$-projective subdynamics of $X$ is $Y$.
		\item The shift action $\sigma$ restricted to $H_1 \times H_2$ is trivial on $X$.
	\end{itemize}
\end{theorem}

\begin{proof}
	Let $S \subset G$ be a finite generating set and consider a recursive bijection $\xi \colon \NN \to S^*$ where $S^*$ is the set of all words over $S$. As $G$ is recursively presented, its word problem $\texttt{WP}(G) = \{w \in S^* \mid w =_G 1_G \}$ is recursively enumerable and there is a Turing machine $\mathcal{M}$ which accepts a pair $(n,n') \in \NN^2$ if and only if $\xi(n) = \xi(n')$ as elements of $G$. For simplicity, fix $\xi(0)$ to be the empty word representing $1_G$.
	
	Let $Y \subset \ag^G$ be the effectively closed $G$-subshift of the statement, we shall encode elements of $\ag$ as binary strings of a fixed length $\kappa \isdef \lceil \log_2(|\ag|)\rceil$. Since $2^{\kappa} > |\ag|$ we may arbitrarily choose a $1$-to-$1$ map $\upsilon\colon \ag \to  \{0,1\}^{\kappa}$ for this encoding. Define the function $\rho \colon Y \to \{0,1\}^{\NN}$ by concatenating all the codings of $y(\xi(i))$ for every $i \in \NN$, that is
	
	$$\rho(y) = \upsilon(y(\xi(0)))\upsilon(y(\xi(1)))\upsilon(y(\xi(2)))\dots$$
	
	Formally, this may be written as
	
	$$\rho(y)_{n} = (\upsilon(y(\xi( \lfloor n/\kappa \rfloor ))))_{n\bmod{\kappa}}.$$
	
	Here $\xi( \lfloor n/\kappa \rfloor) \in S^*$ is identified as an element of $G$. Consider the set $Z = \rho(Y) \subset \{0,1\}^{\NN}$ and the left $G$-action $T \colon G \curvearrowright Z$ defined by $T^g(\rho(y)) \isdef \rho(\sigma^g(y))$. Clearly $\rho$ is a topological conjugacy between $(Y,\sigma)$ and $(Z,T)$. We claim that $(Z,T)$ is an effectively closed $G$-dynamical system.
	
	Indeed, let $w \in \{0,1\}^*$. A Turing machine which accepts $w$ if and only if $[w] \cap Z = \emptyset$ is given by the following scheme: first, note that $Z$ is built from blocks of the form $\upsilon(a)$ for some $a \in \ag$. We detect all $w$ which do not follow that pattern accepting $w$ if for some $n < |w|/\kappa$ we have that $w_{\kappa n},\dots,w_{\kappa n+\kappa-1}$ does not belong to $\upsilon(\ag)$. Second, we must have that if $\xi(n) = \xi(n')$ then the words of length $n$ appearing at both $\kappa n$ and $\kappa n'$ must be identical. We detect the words for which this does not hold by checking for each pair $(\kappa n,\kappa n')$ in the support of $w$ and running $\mathcal{M}$ in parallel over the pair $(n,n')$. If $\mathcal{M}$ accepts for a pair such that $w_{\kappa n},\dots,w_{\kappa n+\kappa-1} \neq w_{\kappa n'},\dots,w_{\kappa n'+\kappa-1}$ then accept $w$. Also, in parallel, use the algorithm recognizing a maximal set of forbidden patterns for $Y$ (this exists by~\cite{ABS2017}, Lemma~2.3) over the pattern coding $$c_w = (\xi(n),\upsilon^{-1}(w_{\kappa n},\dots,w_{\kappa n+\kappa-1}))_{n < |w|/\kappa}.$$ This eliminates all $w$ which codify configurations containing forbidden patterns in $Y$. For the analogous algorithm for $T^s([w])$ just note that as $G$ is recursively presented, the set of pairs $(n,m)$ such that $\xi(n) =_G s\xi(m)$ also form a recursively enumerable set. Therefore $T^s([w])$ also admits the required algorithm.
	
	We use \Cref{simulationTHEOREM} to construct the $(G \times H_1 \times H_2)$-SFT $\Final(Z,T)$. In this case, we shall further restrict $\Final(Z,T)$ with an extra set of forbidden patterns $\FF_4$. Let $B_n$ be the ball of size $n$ in $H_1 \times H_2$ with respect to the metric induced by the set of generators $S_1 \times S_2$ used to construct $\Grid$ in~\Cref{subsecgrid}. Let $p$ be a pattern with support $\{1_G\} \times B_{3^m+1}$, let $(\omega,y) \in [p]$ and $(\omega^0,y^0)= (\omega,y)|_{\{1_G\}\times H_1 \times H_2}$. By definition of $\gamma_m$, we have that for each $(s_1,s_2) \in S_1 \times S_2$, $\gamma_m(1_G, \widehat{\phi}( \mathfrak{C}((\omega^0,y^0),s_1,s_2)))$ only depends on the ball of size $3^m$ around $(s_1,s_2)$. Thus, $\gamma_m$ depends only on $p$. We shall therefore write $\gamma_m(1_G,p,s_1,s_2)$.
	
	For each $m \in \{0,\dots,\kappa-1\}$ we put in $\FF_4$ all the patterns $p$ with support $\supp(p) = \{1_G\} \times B_{3^m+1}$ such that there exists $(s_1,s_2) \in S_1 \times S_2$ satisfying $$\gamma_m(1_G,p,s_1,s_2) \neq \gamma_m(1_G,p,1_{H_1},1_{H_2}).$$
	
	In other words, we force the first $\kappa$ symbols coded in every simulated grid to coincide. As the size of the support of these patterns is bounded, $\FF_4$ is finite and $\widehat{\Final}(Z,T)$ defined by forbidding additionally the patterns in $\FF_4$ is still an SFT. Moreover, the configuration constructed in~\Cref{varphi_onto} clearly satisfies these constraints so the map $\varphi$ is still onto.
	
	Finally, define a map $\hat{\varphi} \colon \widehat{\Final}(Z,T) \to \ag^{G \times H_1 \times H_2}$ by 
	
	$$(\hat{\varphi}(\omega,y)){(g,h_1,h_2)} \isdef \upsilon^{-1}\left( \varphi( \sigma^{(g,h_1,h_2)^{-1}}(\omega,y))|_{0,\dots,\kappa-1} \right).$$
	
	Let $X \isdef \hat{\varphi}(\widehat{\Final}(Z,T))$. The function $\varphi( \sigma^{(g,h_1,h_2)^{-1}}(\omega,y))|_{0,\dots,\kappa-1}$ depends only on a finite support (a ball of size $3^{\kappa}$ around the identity for instance) and clearly commutes with the shift. Therefore $\hat{\varphi}$ is indeed a topological factor and thus $X$ is a sofic subshift. Also, by definition of $\FF_4$ and the fact that $H_1 \times H_2$ is generated by $S_1 \times S_2$ we obtain that $\hat{\varphi}(\omega,y)$ does not depend on $(h_1,h_2)$ and thus $H_1 \times H_2$ acts trivially on $X$. 
	
	Finally, the projective subdynamics $\pi_G(X)$ clearly satisfy that $\pi_G(X) \subset Y$. Let $y \in Y$ and consider $(\omega,r)$ as in \Cref{varphi_onto} such that $\varphi(\sigma^{g}(\omega,r)) = T^g(\rho(y))$. By construction $(\omega,r) \in \widehat{\Final}(Z,T)$ and thus we can furthermore say that $\varphi(\sigma^{(g,h_1,h_2)}(\omega,r))|_{0,\dots,\kappa-1} = T^g(\rho(y))|_{0,\dots,\kappa-1}$. We deduce that
	\begin{align*}
	\hat{\varphi}(\omega,r){(g,h_1,h_2)} & = \upsilon^{-1}\left( \varphi( \sigma^{(g,h_1,h_2)^{-1}}(\omega,r))|_{0,\dots,\kappa-1} \right) \\
	& = \upsilon^{-1}\left(  T^{g^{-1}}(\rho(y))|_{0,\dots,\kappa-1} \right) \\
	& = \upsilon^{-1}\left( (\rho(\sigma^{g^{-1}}(y)))|_{0,\dots,\kappa-1} \right) \\
	& = \upsilon^{-1}\left( \upsilon( y(g)) \right) \\
	& = y(g)
	\end{align*}
	
	Therefore $\pi_G(X) = Y$. As  $H_1 \times H_2$ acts trivially on $X$ every configuration is a periodic extension of some $y \in Y$. Hence the subaction $(X,\sigma_G)$ is conjugate to $(Y,\sigma)$.\end{proof}

\subsection{Strongly aperiodic SFTs in triple direct products}

Next we show how \Cref{theorem_expansive_case} can be applied to produce strongly aperiodic subshifts of finite type. Recall that a $G$-subshift $(X,\sigma)$ is \emph{strongly aperiodic} if the shift action is free, that is, $\forall x \in X, \sigma^g(x) = x \implies g = 1_G$.

\begin{lemma}\label{lemmaperiodic}
	Let $G_i$ for $i \in \{1,2,3\}$ be infinite and finitely generated groups such that there exists a non-empty effectively closed subshift $Y_i \subset \ag^{G_i}$ which is strongly aperiodic. Then $G_1 \times G_2 \times G_3$ admits a non-empty strongly aperiodic SFT.
\end{lemma}

\begin{proof}
	Recall the following general property of factor maps. Suppose there is a factor $\phi\colon (X,T) \twoheadrightarrow (Y,S)$ and let $x \in X$ such that $T^g(x) = x$. Then $S^g(\phi(x)) = \phi(T^g(x)) = \phi(x) \in Y$. This means that if $S$ is a free action then $T$ is also a free action. In particular, it suffices to exhibit a non-empty strongly aperiodic sofic subshift to conclude.
	
	By \Cref{theorem_expansive_case} we can construct for each $i \in \{1,2,3\}$ a non-empty sofic $(G_1 \times G_2 \times G_3)$-subshift $X_i$ whose $G_i$-subaction is conjugate to $Y_i$ and is invariant under the action of $G_j \times G_k$ with $i \notin \{j,k\}$. Let $X = X_1 \times X_2 \times X_3$. We claim that $X$ is a non-empty strongly aperiodic sofic subshift.
	
	Clearly, $X$ is a non-empty sofic subshift. Let $g = (g_1,g_2,g_3) \in G_1 \times G_2 \times G_3$ and $x \in X$ such that $\sigma^{(g_1,g_2,g_3)}(x) = x$. Write $x = (x_1,x_2,x_3)$ and note that
	$$\sigma^{(g_1,g_2,g_3)} = \sigma^{(g_1,1_{G_2},1_{G_3})} \circ \sigma^{(1_{G_1},g_2,1_{G_3})} \circ \sigma^{(1_{G_1},1_{G_2},g_3)}.$$
	
	As $X_i$ is invariant under $G_j \times G_k$, we have that \begin{align*}
	(x_1,x_2,x_3)  & = \sigma^{(g_1,g_2,g_3)}(x_1,x_2,x_3)\\ & = (\sigma^{(g_1,1_{G_2},1_{G_3})}(x_1), \sigma^{(1_{G_1},g_2,1_{G_3})}(x_2) \sigma^{(1_{G_1},1_{G_2},g_3)}(x_3)).
	\end{align*}
	
	On the other hand, as the $G_i$-subaction is conjugate to $Y_i$ which is strongly aperiodic, we conclude that $g_i = 1_{G_i}$. Therefore $g = 1_{G_1 \times G_2 \times G_3}$. As the choice of $x$ was arbitrary, this shows that $X$ is strongly aperiodic.\end{proof}

\Cref{lemmaperiodic} requires the existence of a non-empty effectively closed and strongly aperiodic $G_i$-subshift. Luckily, these objects always exist whenever the word problem of the group is decidable. Furthermore, in the class of recursively presented groups, non-empty effectively closed subshifts which are strongly aperiodic exist if and only if the word problem of the group is decidable. This is proven in~\cite{Jeandel2015} and~\cite{AubrunBarbieriThomasse2015} and can be formally stated as follows.

\begin{lemma}[\cite{AubrunBarbieriThomasse2015}~Theorem~2.8]\label{Lemma_ABT}
	Let $G$ be a recursively presented group. There exists a non-empty, effectively closed and strongly aperiodic $G$-subshift if and only if the word problem of $G$ is decidable.
\end{lemma}

The only if part of this proof is a result by Jeandel~\cite{Jeandel2015} and is basically the fact that a strongly aperiodic SFT (or more generally, an effectively closed strongly aperiodic subshift) in a recursively presented group gives enough information to recursively enumerate the complement of the word problem of the group. Conversely, the existence part of the proof of \Cref{Lemma_ABT} relies on a proof by Alon, Grytczuk, Haluszczak and Riordan~\cite{alonetal_nonrepetitivecoloringofgraphs} which uses Lov\'asz local lemma to show that every finite regular graph of degree $\Delta$ can be vertex-coloured with at most $(2e^{16}+1)\Delta^2$ colours in a way such that the sequence of colours in any non-intersecting path does not contain a square word. Using compactness arguments this result is extended to Cayley graphs $\Gamma(G,S)$ of finitely generated groups where the bound takes the form $2^{19}|S|^2$ colours where $|S|$ is the cardinality of a set of generators of $G$. One can also show that the set of square-free vertex-colourings of $\Gamma(G,S)$ yields a strongly aperiodic subshift, which is thus non-empty if the alphabet has at least $2^{19}|S|^2$ symbols. In the case where $G$ has decidable word problem, a Turing machine can construct a representation of the sequence of balls $B(1_G,n)$ of the Cayley graph and enumerate a codification of all patterns containing a square coloured path.

Adding up \Cref{lemmaperiodic} and \Cref{Lemma_ABT} gives us the following result.

\begin{theorem}\label{corollaryaperioci}
	For any triple of infinite and finitely generated groups $G_1,G_2$ and $G_3$ with decidable word problem, then $G_1 \times G_2 \times G_3$ admits a non-empty strongly aperiodic subshift of finite type.
\end{theorem}

Note that the hypothesis of having decidable word problem is necessary, if not, any finitely generated and recursively presented $G_i$ with undecidable word problem gives a counterexample by \Cref{Lemma_ABT}. On the other hand, to the best of the knowledge of the author, there are no known examples of a group of the form $G_1 \times G_2$ where both $G_i$ are infinite, finitely generated, have decidable word problem, and $G_1 \times G_2$ does not admit a strongly aperiodic SFT. Therefore, it is possible that \Cref{corollaryaperioci} can be improved in that direction.

\subsection{Strongly aperiodic SFTs in branch groups}

Here we exhibit a new class of groups which admit strongly aperiodic SFTs. In particular, this class contains the Grigorchuk group~\cite{Grigorchuk1980}. In order to present this result we need to recall the notion of commensurability.

\begin{definition}
	We say two groups $G,H$ are commensurable if they have isomorphic subgroups of finite index. Namely, $G' \leq G$, $H' \leq H$ such that $[G:G'] < \infty, [H:H'] < \infty$ and $G' \cong H'$.
\end{definition}

A result by Carroll and Penland~\cite{CarrollPenland} establishes that the group property of admitting a non-empty strongly aperiodic SFT is invariant under commensurability. In their article they say that a group $G$ is \emph{weakly periodic} if every non-empty SFT $X \subset \ag^G$ admits a periodic configuration, that is, there exists $x \in X$ and $g \in G \setminus \{1_G\}$ such that $\sigma^g(x) = x$. In other words, a group $G$ is weakly periodic if it does not admit a non-empty strongly aperiodic SFT. Finitely generated free groups are an example of weakly periodic groups~\cite{Piantadosi2008}.

\begin{theorem}[\cite{CarrollPenland}, Theorem~1]\label{teorema_carolpenland}
	Let $G_1$ and $G_2$ be finitely generated commensurable groups. If $G_1$ is
	weakly periodic, then $G_2$ is weakly periodic.
\end{theorem}

With this result in hand, we can show the following:

\begin{lemma}\label{theorema_bonito}
	Let $G$ be a finitely generated group with decidable word problem such that $G$ is commensurable to $G \times G$. Then $G$ admits a non-empty strongly aperiodic subshift of finite type.
\end{lemma}

\begin{proof}
	If $G$ is finite, the result is immediate as $X = \{ x \in \{a,b\}^G \mid |x^{-1}(a)| = 1 \}$ is a non-empty strongly aperiodic subshift of finite type. We suppose from now on that $G$ is infinite. If $G$ is commensurable to $G \times G$, then there exists $H_1 \leq G$, $H_2 \leq G \times G$ of finite index such that $H_1 \cong H_2$. In particular, if we define $\widetilde{H}_1 = G \times H_1$ and $\widetilde{H}_2 = G \times H_2$ we get that $\widetilde{H}_1 \cong \widetilde{H}_2$, $\widetilde{H}_1$ is a finite index subgroup of $G \times G$ and $\widetilde{H}_2$ is a finite index subgroup of $G \times G \times G$. Therefore $G \times G$ and $G \times G \times G$ are commensurable. As commensurability is an equivalence relation in the class of groups we get that $G$ is commensurable to $G \times G \times G$.
	
	As $G$ is infinite, finitely generated and has decidable word problem, \Cref{corollaryaperioci} implies that $G \times G \times G$ is not weakly periodic. It follows by \Cref{teorema_carolpenland} that $G$ is not weakly periodic as well and thus admits a non-empty strongly aperiodic SFT.
\end{proof}

The Grigorchuk group~\cite{Grigorchuk1980} is a famous example of an infinite and finitely generated group of intermediate growth which contains no isomorphic copy of $\ZZ$ and has decidable word problem. It can be defined as the group generated by the involutions $a,b,c,d$ over $\{0,1\}^{\NN}$ as follows, let $x = x_0x_1x_2\dots$ and\begin{align*}
a(x) & = \begin{cases}
1x_1x_2\dots  \mbox{ if } x_0 = 0\\
0x_1x_2\dots  \mbox{ if } x_0 = 1\\
\end{cases} & b(x) & = \begin{cases}
0a(x_1x_2\dots)  \mbox{ if } x_0 = 0\\
1c(x_1x_2\dots)  \mbox{ if } x_0 = 1\\
\end{cases}\\
c(x) & = \begin{cases}
0a(x_1x_2\dots)  \mbox{ if } x_0 = 0\\
1d(x_1x_2\dots)  \mbox{ if } x_0 = 1\\
\end{cases} & d(x) & = \begin{cases}
0x_1x_2\dots\ \ \ \  \mbox{ if } x_0 = 0\\
1b(x_1x_2\dots)  \mbox{ if } x_0 = 1\\
\end{cases}
\end{align*}

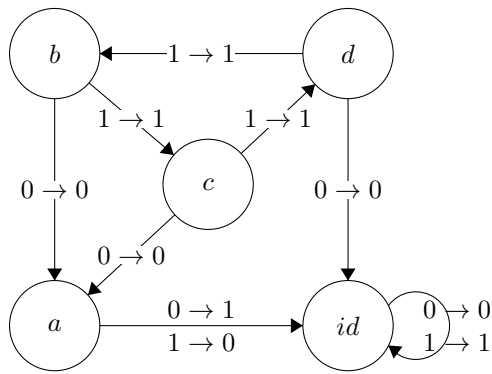
\begin{figure}[h!]
	\centering
	\begin{tikzpicture}[scale=0.2]
	\tikzstyle{every node}+=[inner sep=0pt]
	\draw [black] (20.5,-39.7) circle (3);
	\draw (20.5,-39.7) node {$a$};
	\draw [black] (40,-39.7) circle (3);
	\draw (40,-39.7) node {$id$};
	\draw [black] (20.5,-21.6) circle (3);
	\draw (20.5,-21.6) node {$b$};
	\draw [black] (30.7,-30.3) circle (3);
	\draw (30.7,-30.3) node {$c$};
	\draw [black] (40,-21.6) circle (3);
	\draw (40,-21.6) node {$d$};
	\draw [black] (23.5,-39.7) -- (37,-39.7) node[midway, above = 0.1, fill=white] {$0 \to 1$} node[midway, below= 0.1,fill=white] {$1 \to 0$};
	\fill [black] (37,-39.7) -- (36.2,-39.2) -- (36.2,-40.2);
	\draw [black] (42.68,-38.377) arc (144:-144:2.25);
	\draw (47.25,-39.7) node [right, above = 0.1] {$0 \to 0$};
	\draw (47.25,-39.7) node [right, below = 0.1] {$1 \to 1$};
	\fill [black] (42.68,-41.02) -- (43.03,-41.9) -- (43.62,-41.09);
	\draw [black] (40,-24.6) -- (40,-36.7) node[midway,fill=white] {$0 \to 0$};
	\fill [black] (40,-36.7) -- (40.5,-35.9) -- (39.5,-35.9);
	\draw [black] (37,-21.6) -- (23.5,-21.6) node[midway,fill=white] {$1 \to 1$};
	\fill [black] (23.5,-21.6) -- (24.3,-22.1) -- (24.3,-21.1);
	\draw [black] (22.78,-23.55) -- (28.42,-28.35) node[midway,fill=white] {$1 \to 1$};
	\fill [black] (28.42,-28.35) -- (28.13,-27.45) -- (27.48,-28.21);
	\draw [black] (32.89,-28.25) -- (37.81,-23.65) node[midway,fill=white] {$1 \to 1$};
	\fill [black] (37.81,-23.65) -- (36.88,-23.83) -- (37.57,-24.56);
	\draw [black] (20.5,-24.6) -- (20.5,-36.7) node[midway,fill=white] {$0 \to 0$};
	\fill [black] (20.5,-36.7) -- (21,-35.9) -- (20,-35.9);
	\draw [black] (28.49,-32.33) -- (22.71,-37.67) node[midway,fill=white] {$0 \to 0$};
	\fill [black] (22.71,-37.67) -- (23.63,-37.49) -- (22.96,-36.76);
	\end{tikzpicture}
	\caption{Mealy automaton defining the Grigorchuk group.}
	\label{figure_grigori}
\end{figure}

These four actions can be represented in the Mealy automaton of \Cref{figure_grigori}. Here an arrow of the form $i \to j$ means: ``replace $i$ by $j$ and follow the arrow to the next node''. To compute the image of $x \in \{0,1\}^{\NN}$ under one of these involutions, initialize $n = 0$, start at the node $\texttt{NODE}$ which corresponds to the action and follow the outgoing arrow of the form $x_n \to i$ towards the node it points at. Replace $x_n$ by $i$, $\texttt{NODE}$ by the node the arrow points at and increase $n$ by $1$. The string obtained after iterating this process infinitely many steps is the image of $x$.

Besides the remarkable aforementioned properties, the Grigorchuk group is commensurable to its square.

\begin{lemma}[\cite{ceccherini-SilbersteinC09}~Lemma~6.9.11]
	The Grigorchuk group $G$ is commensurable to $G\times G$.
\end{lemma}

Therefore, we can apply \Cref{theorema_bonito} to obtain:

\begin{corollary}\label{cor_grigo}
	There exists a non-empty strongly aperiodic subshift of finite type defined over the Grigorchuk group.
\end{corollary}

In fact, a similar argument applies to the much larger class of branch groups. An extensive survey about these groups can be found in~\cite{BartholdiGrigorchuk2003Branchgroups}. There is not a unique definition of these groups, we shall work with the following one:

\begin{definition}
	A group $G$ is called a \define{branch group} if there exist two sequences of groups $(L_i)_{i \in \NN}$ and $(H_i)_{i \in \NN}$ and a sequence of positive integers $(k_i)_{i \in \NN}$ such that $k_0 = 1$, $G = L_0 = H_0$ and:
	\begin{enumerate}
		\item $\bigcap_{i \in \NN}{H_i} = 1_G$.
		\item $H_i$ is normal in $G$ and has finite index.
		\item there are subgroups $L_i^{(1)},\dots, L_i^{k(i)}$ of $G$ such that $H_i = L_i^{(1)}\times\dots\times L_i^{k(i)}$ and each of the $L_i^{(j)}$ is isomorphic to $L_i$.
		\item Conjugation by elements of $g$ transitively permutes the factors in the above product decomposition.
		\item $k_{i}$ properly divides $k_{i+1}$ and each of the factors $L_i^{(j)}$ contains $k_{i+1}/k_i$ factors $L_{i+1}^{(j')}$.
	\end{enumerate}
\end{definition}

In particular, the first and second condition say that these are residually finite groups. The third and fifth conditions imply that they are infinite. Note that there are examples where these groups might not be finitely generated~\cite{BartholdiGrigorchuk2003Branchgroups}~Proposition 1.22, or where they might have undecidable word problem, see~\cite{Grigorchuk1984} or~\cite{BartholdiGrigorchuk2003Branchgroups}~Theorem 3.1. We can characterize all finitely generated recursively presented branch groups which have decidable word problem by the existence of strongly aperiodic SFTs. What is more, we only use the second, third and fifth conditions in our proof.

\begin{theorem}\label{teorema_branch}
	Let $G$ be a finitely generated and recursively presented branch group. Then $G$ has decidable word problem if and only if there exists a non-empty strongly aperiodic $G$-SFT.
\end{theorem}

\begin{proof}
	As $G$ is recursively presented, \Cref{Lemma_ABT} implies that if $G$ admits a non-empty strongly aperiodic $G$-SFT then it has decidable word problem. Conversely, let $(L_i)_{i \in \NN}$, $(H_i)_{i \in \NN}$ and $(k_i)_{i \in \NN}$ be the sequences associated to $G$. Recall that $H_i = L_i^{(1)} \times \dots \times L_i^{(k_i)}$ where each $L_i^{(j)}$ is isomorphic to $L_i$. As $H_i$ is a finite index subgroup of $G$, it is infinite and finitely generated. Moreover, as each $H_i$ is a finite direct product of $L_i$, each $L_i$ is also infinite and finitely generated. As any finitely generated subgroup of a group with decidable word problem also has decidable word problem, we conclude that $L_i$ is an infinite, finitely generated group with decidable word problem.
	
	By definition $H_2 = L_2^{(1)} \times \dots \times L_2^{(k_2)}$ and as $k_{i}$ properly divides $k_{i+1}$ we have that $k_2 \geq 4$. Applying~\Cref*{corollaryaperioci} shows that $H_2$ admits a non-empty strongly aperiodic $H_2$-SFT. Finally, as $H_2$ is a subgroup of finite index of $G$, \Cref{teorema_carolpenland} implies that $G$ admits a non-empty strongly aperiodic $G$-SFT.\end{proof}

Examples of finitely generated branch groups with decidable word problem can be found in~\cite{BartholdiGrigorchuk2003Branchgroups}. The Grigorchuk group is an example, but general $\mathsf{G}$-groups and $\mathsf{GGS}$-groups and more generally, finitely generated spinal groups given by recursive sequences also satisfy those properties and thus admit non-empty strongly aperiodic SFTs.

\section*{Acknowledgments} %
The author wishes to thank two anonymous referees who suggested style improvements and corrected several typos. This research was mainly carried out while the author was affiliated to ENS de Lyon in France.

\bibliographystyle{amsplain}

\begin{dajauthors}
\begin{authorinfo}[pgom]
  Sebasti\'an Barbieri\\
  Department of Mathematics\\
  University of British Columbia\\
  Vancouver, Canada\\
  sbarbieri\imageat{}math\imagedot{}ubc\imagedot{}ca \\
  \url{https://www.math.ubc.ca/\textasciitilde sbarbieri}
\end{authorinfo}
\end{dajauthors}

\end{document}